\documentclass[a4paper,12pt]{article}
\usepackage[a4paper,margin=1in]{geometry}
\usepackage[utf8]{inputenc}
\usepackage[T1]{fontenc}
\usepackage[leqno]{amsmath}
\usepackage{amssymb}
\usepackage{amsthm}
\usepackage{mathtools,thmtools}
\usepackage{placeins}
\usepackage[english]{babel}
\usepackage{graphicx}
\usepackage{enumerate}
\usepackage{enumitem}
\usepackage{multirow}
\usepackage[autostyle=true,german=quotes]{csquotes}
\usepackage{color}
\usepackage{esint}
\usepackage[hidelinks]{hyperref}
\usepackage{cleveref}

\numberwithin{equation}{section}


\date{}

\crefname{subsection}{subsection}{subsections}
\crefname{subsubsection}{subsubsection}{subsubsections}

\makeatletter
\newcommand{\myitem}[1]{%
\item[#1]\protected@edef\@currentlabel{#1}%
}
\makeatother

\newlist{thmlist}{enumerate}{1}
\setlist[thmlist]{label=(\alph{thmlisti}), ref=\alph{thmlisti}}
\newlist{lemlist}{enumerate}{1}
\setlist[lemlist]{label=(\alph{lemlisti}), ref=\alph{lemlisti}}

\crefname{subsection}{subsection}{subsections}
\crefname{subsubsection}{subsubsection}{subsubsections}
\crefname{figure}{figure}{figures}

\DeclareRobustCommand{\abbrevcrefs}{%
\crefname{section}{sec.}{secs.}%
\crefname{subsection}{subsec.}{subsecs.}%
\crefname{theorem}{thm.}{thms.}%
\crefname{lemma}{lem.}{lems.}%
\crefname{corollary}{cor.}{cors.}%
\crefname{proposition}{prop.}{props.}%
\crefname{figure}{fig.}{figs.}%
\crefname{equation}{eqn.}{eqns.}%
}
\DeclareRobustCommand{\cshref}[1]{{\abbrevcrefs\cref{#1}}}

\let\originalleft\left
\let\originalright\right
\renewcommand{\left}{\mathopen{}\mathclose\bgroup\originalleft}
\renewcommand{\right}{\aftergroup\egroup\originalright}

\renewcommand{\d}{\,\mathrm{d}}
\newcommand{\pw}{{\rm pw}}
\newcommand{\nc}{{\rm nc}}

\newcommand{\ut}{\widehat u}
\newcommand{\vt}{\widehat v}

\newcommand{\uM}{u_\M}

\newcommand{\vM}{v_\M}
\newcommand{\wM}{w_\M}

\newcommand{\mut}{\widehat{\mu}}

\newcommand{\GGamma}{G}

\renewcommand{\div}{\mathrm{div}}
\DeclareMathOperator{\curl}{Curl}

\newcommand{\Inc}{I}
\newcommand{\trb}[1]{|\!|\!|#1|\!|\!|}
\newcommand{\trbpw}[1]{\trb{#1}_\pw}
\newcommand{\jump}[1]{\left[#1\right]}

\newcommand{\id}{\mathrm{id}}

\newcommand{\R}{{\mathbb{R}}}
\newcommand{\N}{{\mathbb{N}}}

\newcommand{\Vnc}{V_{\mathrm{nc}}}
\newcommand{\Vh}{\Vnc}
\newcommand{\Vt}{V_\pw}

\newcommand{\M}{\mathrm{M}}

\newcommand{\TT}{\mathbb{T}}
\newcommand{\T}{\mathcal{T}}
\newcommand{\E}{\mathcal{E}}

\newcommand{\vpw}{ v_\pw}

\newcommand{\wt}{\widehat w}

\newcommand{\bht}{\widehat{\beta}_h}
\newcommand{\btT}{\widehat{\beta}_0}
\newcommand{\mtT}{\widehat{\mu}}
\newcommand{\vrpT}{\varrho_{\mathrm{uq}}}
\newcommand{\vrmT}{\varrho_{\mathrm{ex}}}

\theoremstyle{remark}

\theoremstyle{definition}
\newtheorem{defn}{Definition}[section]

\theoremstyle{plain}
\newtheorem{theorem}{Theorem}[section]
\newtheorem{lemma}[theorem]{Lemma}

\newtheorem{proposition}[theorem]{Proposition}

\theoremstyle{remark}
\newtheorem{remark}[theorem]{Remark}

\newcommand{\Const}[1]{C_{\rm #1}}
\newcommand{\Constb}[1]{C_{\beta,{\rm #1}}}

\newcommand{\LJ}{\Lambda_{\rm J}}

\newcommand{\Cqb}{C_{\rm qb}}

\newcommand{\LDG}{L}

\usepackage{totcount}
\newtotcounter{nconst}
\makeatletter
\newcount\rc@count
\let\rc@clearconstantlist\empty

\newcommand\rc@clearconstant[1]{\global\expandafter\let\csname rc@const@#1\endcsname\undefined}

\newcommand\resetconstants[1]{%
    \def\rc@constname{#1}% Set the new base name of the constants to the argument
    \global\rc@count=1\relax % Reset the constant counter to 1
    \bgroup 
        \let\\\rc@clearconstant % map over the list of constants that have been defined, clearing each of them.
        \rc@clearconstantlist
        \global\let\rc@clearconstantlist\empty % Globally empty the list of constants.
    \egroup
}
\newcommand\const[1]{%
    \@ifundefined{rc@const@#1}{%
        \expandafter\xdef\csname rc@const@#1\endcsname{%
           \noexpand\rc@useconst{\rc@constname}{\the\rc@count}%
        }%
        \g@addto@macro\rc@clearconstantlist{\\{#1}}%
		\global\setcounter{nconst}{\the\rc@count}\relax
        \global\advance\rc@count1\relax
    }{}%
    \csname rc@const@#1\endcsname
}

\newcommand\rc@useconst[2]{\ensuremath{#1_{#2}}}
\makeatother

	\resetconstants{C}

\makeatletter
\newcommand{\leqnomode}{\tagsleft@true}
\newcommand{\reqnomode}{\tagsleft@false}
\makeatother

\title{%
	Guaranteed inf--sup bounds and existence verification for semilinear elliptic problems via nonconforming finite elements
}
\author{Benedikt Gr\"a{\ss}le\footnote{Institut für Mathematik, Universität Zürich, Winterthurerstr.~190, CH-8057 Zürich, Switzerland \\
  (benedikt.graessle@math.uzh.ch)}
}
\begin{document}

\maketitle

\begin{abstract}
	A Newton--Kantorovich-type argument enables the a posteriori existence verification of a locally unique regular root near a computed approximation.
	This framework allows for non-selfadjoint problems and extends the existing verification theory to nonconforming discretisations.
	A key ingredient is a guaranteed lower bound for the continuous inf--sup constant of the linearisation obtained with a novel approximation error bound for nonconforming schemes.
	All quantities are obtained from a postprocessing on the same discretisation within the adaptive loop.
	The theory is applied to a fourth-order formulation of the stationary 2D Navier--Stokes equations and illustrated by numerical experiments.
\end{abstract}



\section{Introduction}\label{sec:Introduction}
Nonlinear problems may admit multiple solutions, no solution, or singular configurations such as bifurcation points.
That raises a basic difficulty for their numerical discretisation:
a computed approximation may converge to a different solution branch, oscillate between nearby solutions, or may fail to reflect the local structure of the solution set. 
Consider a semilinear elliptic model problem with
$m$-harmonic principal part
\begin{align*}
	(-\Delta)^m u + G(u) = F.
\end{align*}
The focus is on the Laplacian for $m=1$ and the biharmonic operator for $m=2$
that define the energy scalar product $a(\bullet,\bullet)=(D^m\bullet,D^m\bullet)_{L^2(\Omega)}$
in the Hilbert space $V\coloneqq H^m_0(\Omega)$.
Given $F\in V^*$, the weak formulation seeks $u\in V$ with
\begin{align}\label{eqn:SWP}
	N(u;\varphi)\coloneqq a(u,\varphi) + G(u;\varphi) - F(\varphi)=0\qquad\text{for all }\varphi\in V.
\end{align}
Here the semilinear nature is modelled by the assumption that $G\in C^1(V;V^*)$ is a compact nonlinear operator, and the solutions are exactly the roots $u\in V$ of $N(u)=0$.

Since the semilinear problem~\eqref{eqn:SWP} is not well posed in the sense of a globally unique solution, the natural first target for the numerical analysis is a local solution concept with additional regularity properties.
A root \(u\) of~\eqref{eqn:SWP} is called \emph{regular} if the Fréchet derivative \(DN(u)\) is bijective and hence has a positive inf--sup constant \(\beta>0\).
These roots are always isolated by the inverse function theorem, locally stable under perturbations, and away from bifurcations.
In fact, this regularity property covers the generic situation.

The numerical analysis for the approximation of regular roots is well developed in the following sense.
Under the typical assumption that a regular root $u\in V$ of~\eqref{eqn:SWP} exists \emph{sufficiently close} to the discrete approximation $u_h\in V_h$, one obtains classical a priori and a posteriori error estimates~\cite{CGK:GeneralizedKarmanEquations2007,MN:ConformingFiniteElement2016,BNRS:C0InteriorPenalty2017,KPP:MorleyFiniteElement2021,CNRS:UnifiedPrioriAnalysis2023,CGN:PosterioriErrorControl2024}, and even optimal convergence results for adaptive schemes in special cases~\cite{CN:AdaptiveMorleyFEM2021,CG:AdaptiveMorleyFEM2025}.
However, the regularity assumption is crucial and hides a critical dependence on the inf--sup constant $\beta$ of the linearised equation.
This number governs not only the constants in a priori and a posteriori error estimates, but also the sensitivity of the root under perturbations and the size of the neighbourhood in which Newton-type arguments and local discrete solvability results apply.
Small values of $\beta$ indicate ill-conditioning and may occur near singular configurations such as bifurcation or turning points.
In this regime the error bounds deteriorate and the analysis may require a prohibitively fine mesh, see~\cite[rem.~6.4]{CG:AdaptiveMorleyFEM2025} for a concrete discussion. 
Since the existence and regularity of roots are generally unclear, this raises a basic practical question: Are there regular roots with moderate inf--sup constant $\beta$ for a given source $F\in V^*$?

Computer-assisted existence proofs address this issue from a different viewpoint and have been developed over the past decades~\cite{NHW:NumericalMethodVerify2005,TLO:VerifiedComputationsSemilinear2013,NPW:NumericalVerificationMethods2019}, with earlier roots in~\cite{Nak:NumericalApproachProof1988,Plu:ComputerassistedProofsSemilinear2009}.
They typically employ a Newton--Kantorovich or fixed-point argument on the continuous level at a computed approximation $\widehat u\in V$, and thereby turn numerical data into a mathematical existence proof in the sense of validated numerics.
For semilinear problems, such verification methods are often realised through conforming finite-dimensional projections and require explicit a priori bounds for auxiliary linear problems that often have limited regularity.
While this is effective for many second-order problems, it becomes restrictive for higher-order equations ($m\geq2$) where conforming finite elements are complicated and expensive in the absence of an adaptive theory.

This paper develops a verification framework based on nonconforming discretisations that include (enriched) Crouzeix--Raviart ($m=1$) and Morley FEM ($m=2$).
It turns quasi-optimal nonconforming approximations into computable stability information and existence certificates on the same discretisation used in the adaptive loop.
The analysis considers exact discrete quantities; the influence of finite arithmetic is not discussed here, and the reader is referred to~\cite{NPW:NumericalVerificationMethods2019} and the references therein for validated numerics.

\subsection*{Main results}%
\label{sub:Main results}
\begin{enumerate}[label=(\Alph*)]
	\item\label{it:A}
	A genericity result justifies the focus on regular roots as the natural first target 
	for weakly coercive problems with compact semilinearity $G\in C^1(V;V^*)$:
	For a dense open set $\mathcal{O}\subset V^*$ of sources $F\in\mathcal{O}$, all roots of~\eqref{eqn:SWP} are regular and their number is finite.
	This general characterisation of solution sets for semilinear problems 
	implies that singular situations such as bifurcations are \emph{rare}.
	It extends related genericity results~\cite{FP:ComportementGlobalSolutions1967,FT:StructureSetStationary1977} on the Navier--Stokes equations to general semilinear problems.

\item \label{it:B}
	The existence verification theory is extended to quasi-optimal nonconforming schemes $V_\nc\not\subset V$ motivated by recent insights into their algebraic structure~\cite{VZ:QuasioptimalNonconformingMethods2019,CNRS:UnifiedPrioriAnalysis2023,CGN:PosterioriErrorControl2024}.
	These schemes come naturally with a computable smoothing operator $J:V_\nc\to V$
	that enables a Newton--Kantorovich argument formulated at the conforming counterpart $Jv_\nc\in V$ of the discrete approximation $v_\nc\in V_\nc$.
	This leads to computable conditions \ref{ass:N} for the existence of a regular root $u\in V$ in an explicit neighbourhood of $v_\nc$ together with a guaranteed lower bound on $\beta>0$.
	An advantage of this nonconforming approach is that all numerical quantities are evaluated by postprocessing $v_\nc$ within the adaptive loop.

\item\label{it:C}
	A central task for the Newton--Kantorovich argument in \ref{it:B} is the computation of a guaranteed lower bound on the inf--sup constant for $DN(Jv_\nc)$.
	The key novelty is a computable approximation bound for the nonconforming scheme applied to the linear \(m\)-harmonic problem with right-hand side in \(H^{-s}(\Omega)\) for fixed \(0\le s\le m\).
	This bound controls perturbations of \(DN(Jv_\nc)\) where $s$ is such that
	\begin{align*}
        DG(Jv_\nc)\in L(V;H^{-s}(\Omega)).
	\end{align*}
	An improved Sobolev regularity $s<m$ may follow either from $G\in C^1\bigl(V;H^{-s}(\Omega)\bigr)$ or from additional regularity of $Jv_\nc$ (typically $Jv_\nc\in V\cap W^{m,\infty}(\Omega)$) in applications.

	\item The abstract existence verification framework is applied to a Navier--Stokes-type model problem with source $F\in L^2(\Omega)$ in a polygonal Lipschitz domain $\Omega\subset \R^2$: 
		The stationary Navier--Stokes equations allow for an equivalent formulation as a fourth-order problem for the stream-function $u\in V\equiv H^2_0(\Omega)$ satisfying \begin{align}\label{eqn:intro_NVS_SP}
			\begin{aligned}
				\Delta^2u -\div(\Delta u \curl u)&= F
			\end{aligned}\qquad
			\begin{aligned}
				&\text{in }\Omega\subset \R^2
			\end{aligned}
		\end{align}
		with the rotated gradient $\curl v\coloneqq
		(\partial v/\partial y;-\partial v/\partial x)$.
		This non-selfadjoint fourth-order problem is a natural test case for nonconforming discretisations that avoid complicated continuity requirements.
		All verification quantities are computed for the Morley FEM, and numerical benchmarks certify existence and local uniqueness.
\end{enumerate}

\subsection*{Outline}%
\label{sub:Overview}

\Cref{sec:Existence verification for fourth-order problems} develops the
abstract existence verification within a general framework to underline its algebraic
nature and universality. The semilinear problem is
formulated in a Hilbert space setting, and
\cref{thm:reg_roots} provides the genericity of regular roots with respect to
the right-hand side. \Cref{sub:Nonconforming reference discretisation}
introduces the abstract nonconforming setting and its algebraic ingredients,
namely the interpolation operator $I$ and the smoother $J$.
For completeness, a typical a priori convergence analysis towards a regular root that extends~\cite{CNRS:UnifiedPrioriAnalysis2023} is presented in \cref{thm:abstract_apriori} with proof given in \cref{apx:Proof of a priori}.

The remaining parts of this paper consider the task of verifying existence and stability properties of regular roots as part of the original discretisation process.
\Cref{sub:The Newton--Kantorovich theorem} derives a Newton--Kantorovich verification criterion \ref{ass:N} for a nonconforming approximation $v_\nc\in V_\nc$ and its conforming representative $Jv_\nc\in V$, while \cref{sub:Morley_reference_discretisation} provides the required residual control.
\Cref{sec:The_continuous_inf--sup_constant} concerns the guaranteed lower bound on the continuous inf--sup constant of $DN(Jv_\nc)$ based on a novel approximation control for nonconforming schemes.
\Cref{thm:continuous_inf_sup} provides the inf--sup bound and \cref{sub:computation_kappaM_Ts} identifies the required discrete quantities as
finite-dimensional algebraic (eigenvalue) problems.

\Cref{sec:Algebraic_realisation_and_application} applies the abstract theory
to the stream-function formulation of the stationary two-dimensional
Navier--Stokes equations. \Cref{sub:Stream-function vorticity formulation of Navier--Stokes} verifies the genericity result in this setting,
\cref{sub:Morley} specialises the framework to the Morley discretisation,
and \cref{sub:Explicit constants} derives the computable constants required
by the verification theorem and the guaranteed inf--sup bound. 
Numerical experiments on the resulting computer-assisted existence and local uniqueness verification, neglecting numerical errors from finite precision, conclude this paper.

\subsection*{General notation}%
\label{sub:General notation}
Throughout this paper, $\T$ denotes a shape-regular triangulation into compact triangles 
$T\in\T$ with positive area $|T|>0$, diameter $h_T=\mathrm{diam}(T)$, centroid $\mathrm{mid}(T)$,
and outer unit normal $\nu_T$ on $\partial T$.
The set of admissible triangulations $\TT$ consists of all possible refinements $\T$ generated by
successive newest-vertex bisections (NVB) from a fixed shape-regular initial
triangulation $\T_0$ of the domain $\Omega$.
Let $\mathcal{E}(T)$ and $\mathcal{V}(T)$ denote the sets of edges and vertices of $T\in\mathcal{T}$.
All (resp.~interior,~boundary) vertices in $\T$ are collected in the set $\mathcal{V}$ (resp.~$\mathcal{V}(\Omega), \mathcal V
(\partial\Omega)$).
Similarly, $\mathcal{E}$ (resp.~$\mathcal{E}(\Omega), \mathcal E(\partial\Omega)$) 
denotes the set of all (resp.~interior,
boundary) edges.
Each edge $E\in\E$ is associated with unit normal and tangent vectors of fixed orientation, 
denoted $\nu_E$ and $\tau_E$.
With this sign convention, 
the jump $[v]_E$ across $E$ of a piecewise Lipschitz continuous function $v$ 
across $E$ reads
\begin{align*}
	\jump v(x) &\coloneqq
\begin{cases}  
(v|_{T_+})(x) - (v|_{T_-})(x)  &  \text{at}\,\, x \in E=\partial T_{+}\cap \partial T_{-} \in \mathcal{E}(\Omega),\\ 
 v(x) & \text{at}\,\, x\in E \in \mathcal{E}(\partial \Omega) .
\end{cases}%
\end{align*}%
The space $P_{\hspace{-.13em}k}(T)$ of polynomials of total degree at most $k\in\mathbb N_0$ 
in a triangle $T\in\T$ defines
the space of piecewise polynomials
\begin{align*}
	P_{\hspace{-.13em}k}(\T)\coloneqq\{p\in L^\infty(\Omega) : p|_{T}\in
	P_{\hspace{-.13em}k}(T)\text{ for all } T\in \T\}.
\end{align*}
The $L^2$ orthogonal projection $\Pi_0:L^2(\Omega)\to P_0(\mathcal{T})$ onto piecewise constants
coincides with the piecewise integral mean and applies
componentwise to vectorvalued functions.
The maximal mesh size $h_{\max}=\max_{T\in\mathcal{T}}h_{T}$
controls $h_\T\in P_{\hspace{-.13em}0}(\T)$ with $h_\T|_T\equiv
h_T$ on any $T\in\T$ and
the Poincar\'e constant $\Const{P}\leq 1/\pi$ satisfies 
$\|(1-\Pi_0)v\|_{L^2(T)}\leq \Const{P}h_T\|\nabla v\|_{L^2(T)}$ for any $v\in H^1(T)$.
Let $\nabla_\pw,\curl_\pw, D^2_\pw$, 
and $\Delta_\pw$ denote the piecewise action of the corresponding 
differential operator over $\mathcal{T}$.
The notation $A\lesssim B$ suppresses a generic constant $C$ independent of the mesh sizes
in $A\leq CB$ and $A\approx B$ abbreviates $A\lesssim B\lesssim A$.%

\section{Existence verification of regular roots}%
\label{sec:Existence verification for fourth-order problems}
The Newton--Kantorovich theorem guarantees the existence of regular roots 
of semilinear problems in the neighbourhood of some nonconforming approximation.

\subsection{Regular roots of semilinear problems}%
\label{sub:Regular roots of semilinear problems}
This section considers an abstract Hilbert space 
$(V, a)$ with induced norm $\trb{\bullet}$ %
and associated operator norm $\trb{\bullet}_*$ in $V^*$.
The abstract semilinear problem for a compact $G\in C^1(V;V^*)$
with source $F\in V^*$ seeks a root $u\in V$ of $N\in C^1(V;V^*)$ given by
\begin{align}\label{eqn:ASWP}
	N(v)\coloneqq a(v,\bullet) + G(v) - F\in V^*.
\end{align}

A root $u\in V$ of \eqref{eqn:ASWP} is called a \emph{regular root} of $N(u)=0$,
if the Fr\'echet derivative $DN(u)\in L(V;V^*)$ is bijective.
A regular root $u\in V$ is always isolated by the inverse function theorem and leads to a positive inf--sup constant 
\begin{align}\label{eqn:cont_inf_sup}
	0<\beta:=\inf_{0\ne v\in V}\frac{\trb{DN(u;v)}_{*}}{\trb{v}}
	=\inf_{0\ne v\in V} \sup_{0\ne\varphi\in V}\frac{a(v,\varphi)+DG(u;v,\varphi)}{\trb v\trb{\varphi}}.
\end{align}
\begin{lemma}[inf--sup condition~{\cite{Bra:FiniteElementsTheory2007,BS:MathematicalTheoryFinite2008}}]\label{lem:cont_inf_sup}
	If $DN(u)\in L(V;V^*)$ is bijective at $u\in V$, it satisfies the inf--sup condition~\eqref{eqn:cont_inf_sup} for
	$\beta>0$ with $\|DN(u)^{-1}\|_{L(V^*;V)}=\beta^{-1}$.
\end{lemma}
\begin{proof}
	The proof of this standard result is omitted.
\end{proof}
The converse of \cref{lem:cont_inf_sup} is also true by the Fredholm alternative
for the compactly perturbed isomorphism $DN(u)$.
Let $A\in L(V;V^*)$ denote the isometry induced by the scalar product $a(\bullet,\bullet)$.
\begin{theorem}[genericity of regular roots]\label{thm:reg_roots}
	Assume $T\coloneqq A+G$ is weakly coercive, i.e.,
	\begin{align*}
		\trb{T(v)}_* \to \infty\quad\text{as}
		\quad \trb{v}\to\infty\quad \text{with }v\in V.
	\end{align*}
	Let $\mathcal{O}\subset V^*$ denote the set of right-hand sides 
	$F\in V^*$ such that all roots $u\in V$ of~\eqref{eqn:ASWP} are regular.
	Then $\mathcal{O}$ is a dense open subset of $V^*$ and, 
	for any $F\in\mathcal{O}$, there are at most finitely many (possibly zero)
	roots of~\eqref{eqn:ASWP}, i.e., $|T^{-1}(F)|<\infty$ if $F\in\mathcal{O}$.
\end{theorem}
\newcommand{\regval}{\mathrm{reg}}
\begin{proof}
	The proof consists of $3$ steps.

	\medskip
	\noindent\emph{Step 1} is the reduction to regular values of the nonlinear Fredholm operator $T=A+G$.
	The isomorphism $A\in L(V;V^*)$ is injective and surjective so that $\textrm{ker}(A) = 0 =
	\textrm{codim}(\textrm{im}(A))$.
	Hence $A$ is a linear Fredholm operator of index $\mathrm{ind}(A) = 0$.
	Since $G\in C^1(V;V^*)$ is compact,
	the compact perturbation $T=A+G$
	is a nonlinear Fredholm operator with the same index 
	$\mathrm{ind}(T) =\mathrm{ind}(A) = 0$
	(cf.~\cite[ex.~8.16]{Zei:NonlinearFunctionalAnalysis1992} for details).

	An element $F\in V^*$ is a regular value of $T\in C^1(V;V^*)$,
	if the Fr\'echet derivative $DT(v)\in L(V;V^*)$ 
	is surjective at all $v\in T^{-1}(F)$. 
	Since the Fredholm index of $T$ is $0$,
	$DT(v)=DN(v)$ (from $T-N=F\in V^*$ constant)
	at $v\in V$ is surjective if and only if it is bijective.
	Observe that for any $F\in V^*$, 
	$u\in V$ is a root of $N$ if and only if $u$ satisfies $T(u)=F$.
	Hence the set of all regular values
	\[
	\textrm{reg}(T)\coloneqq\{F\in V^*\ :\ DT(v)\text{ is surjective for all }v\in T^{-1}(F)\}
	\]
	of $T$ are exactly the right-hand sides $F\in V^*$ such that all roots of 
	$N$ are regular.

	\medskip
	\noindent\emph{Step 2} applies the theorem of
	Sard--Smale~\cite{Sma:InfiniteDimensionalVersion1965,Geb:LeraySchauderDegreeFramed1975,QS:HausdorffConullityCritical1972},
	which states that the set of regular values of any proper Fredholm map of index $0$
	is a dense open subset of $V^*$.
	It remains to verify that $T$ is proper, i.e., the preimage 
	$T^{-1}(K)$ of any compact set $K\subset V^*$ is compact in $V$.
	The weak coercivity of $T$ implies that $T^{-1}(C)$ is bounded in $V$ for
	any bounded set $C\subset V^*$. 
	In fact, the reverse implication also 
	holds~\cite[p.~173]{Zei:NonlinearFunctionalAnalysis1992}.

	To show that $T$ is proper, 
	let $K\subset V^*$ be any compact set and consider an arbitrary sequence $(u_j)_{j\in\N}$ in the preimage $T^{-1}(K)$.
	Because $K$ is compact, there exists a convergent subsequence (not relabelled) such that $f_j\coloneqq T(u_j)$
	converges to some $f\in K$.
	Decompose the image of the sequence as $f_j=a_j + k_j$ with $a_j\coloneqq Au_j$ and $k_j\coloneqq G(u_j)$.
	Since $G$ is a compact operator and $T^{-1}(K)$ is bounded by the weak coercivity of $T$, there exists a further subsequence (not relabelled)
	such that $k_j\equiv G(u_j)$ converges to some $k\in V^*$.
	Consequently, $a_j=f_j - k_j$ converges to $f-k$ as $j\to \infty$.
	As $A$ is an isomorphism, its inverse is continuous so that $u_j\equiv A^{-1}a_j$ converges to 
	$u\coloneqq A^{-1}(f-k)$.
	Since this holds for any compact set $K\subset V^*$, the operator $T$ is proper.
	Recall that $T$ is a nonlinear Fredholm operator of index $0$ from step 1.
	Hence the theorem of Sard--Smale 
	provides that the set of regular values $\textrm{reg}(T)$ of $T$ is open and
	dense.

	\medskip
	\noindent\emph{Step 3} concludes the proof.
	Let $F\in \textrm{reg}(T)\subset V^*$ be a regular value of $T$.
	Then for any $v\in T^{-1}(F)$, $DT(v)$ is surjective and (by step 1) bijective.
	The inverse function theorem \cite[thm.~4.F]{Zei:NonlinearFunctionalAnalysis1992} provides an open neighbourhood
	$V(v)\subset V$ around $v\in V$ such that $T|_{V(v)}:V(v)\to V^*$ is bijective
	onto its image.
	Consequently, $v$ is the only solution of $T(v) = F$ in $V(v)$.
	This holds for all $v\in V$ and implies that $T^{-1}(F)$ contains only isolated points.
	Since the map $T$ is proper by step 2, the preimage $T^{-1}(F)$ is also compact.
	As a compact set of isolated points, $T^{-1}(F)$ is finite.
	By step 2, this reveals at most finitely many roots to $N=T-F$
	for the right-hand side $F$.
	This concludes the proof for the open and dense set 
	$\mathcal{O}\coloneqq \textrm{reg}(T)$ from step 3.
\end{proof}

\begin{remark}[weak coercivity]\label{rem:weak_coercivity}
	The weak coercivity required in~\cref{thm:reg_roots} holds for a large class of nonlinear maps.
	In particular, operators $T\in C^1(V;V^*)$ related to differential equations usually 
	satisfy a priori bounds of the form
	\begin{align*}
		\trb{v}\leq C\trb{T(v)}_*\quad\text{for all }v\in V
	\end{align*}
	with some constant $C$.
	A priori bounds and (strong) coercivity each imply weak coercivity.
\end{remark}
\begin{remark}[extensions to~\cref{thm:reg_roots}]\label{rem:reg_roots}
	\Cref{thm:reg_roots} holds in more generality.
	Without the weak coercivity of $A+G$ and provided that $V$ is separable, 
	one still obtains the existence of a dense set 
	$\mathcal{D}\subset V^*$ such
	that all roots of $N=A+G-F$ are regular for any $F\in\mathcal{D}$.
	However, $\mathcal{D}$ does not need to be open, see
	\cite{Sma:InfiniteDimensionalVersion1965,Zei:NonlinearFunctionalAnalysis1997} 
	for further details.

	The result also extends to the Banach setting, i.e., to compact nonlinear perturbations $T=A+G\in C^1(X;Y)$ of an isomorphism $A\in L(X;Y)$ between Banach spaces $X,Y$.
\end{remark}

\subsection{Quasi-optimal nonconforming discretisation}%
\label{sub:Nonconforming reference discretisation}
Recent progress in the understanding of (quasi-optimal) nonconforming schemes $V_\nc\not\subset V$
involves a careful analysis based on the algebraic interplay of the natural interpolation $I:V\to V_\nc$
and a right-inverse with additional properties $J:V_\nc\to V$, called smoother.
Following the abstract description of nonconforming schemes 
in~\cite[sec.~2]{CG:OptimalConvergenceRates2025}, let the 
energy scalar product $a(\bullet,\bullet)=a_\pw(\bullet,\bullet)|_{V\times V}$ extend to a semi-scalar product 
$$
a_\pw(\bullet,\bullet):V_\pw\times V_\pw\to \R
$$
on the superspace $V_\pw\coloneqq V+ V_\nc$ and denote the induced (semi-)norm by $\trbpw{\bullet}$.
On $V_\nc$, $a_\pw(\bullet,\bullet)$ is a scalar product and $\trbpw{\bullet}$ a norm.
The subscript $(\bullet)_{\pw}$ indicates that the extension is usually obtained by piecewise
action of differential operators in applications.

Classical nonconforming FEM induce an interpolation operator $I\in L(V;V_\nc)$
with approximation properties and an important orthogonality: It is 
the Galerkin projector
\begin{align}\label{eqn:I1}%
	a_{\pw}(v-I v, v_\nc) = 0
	\qquad\text{for all }(v,v_{\mathrm{nc}})\in V\times V_\nc.
\end{align}
A direct consequence is the contraction property $\|I\|\leq 1$ and the Pythagoras identity  
\begin{align*}
    \trb{ v }_{\rm pw}^2
    =   \trb{ v - Iv }_{\rm pw}^2 + \trb{ Iv}_{\rm pw}^2
    \quad\text{for all } v\in V.
\end{align*}
The smoother $J:V_\nc\to V$ is a right-inverse of the interpolation, namely
\begin{align}\label{eqn:JI}%
	I J v_\nc = v_\nc\qquad\text{for all }v_\nc\in V_\nc,
\end{align}
and quasi-optimal in the sense that there exists some $\Lambda_{J}>0$ with
\begin{align}\label{eqn:J_qo}%
	\trb{v_\nc-Jv_\nc}_\pw
	\leq \Lambda_{J}\min_{v\in V}\trb{v-v_\nc}_\pw\qquad\text{for all }v_\nc\in V_\nc.
\end{align}
In other words $Jv_\nc$ is a quasi-best-approximation of $v_\nc\in V_\nc$ in $V$. 
Observe that~\eqref{eqn:I1} implies $I=\id$ on $V\cap V_\nc$ so that $J=\id$ on $V\cap V_\nc$
is a natural requirement. 
It is shown in~\cite[lem.~2.2]{CN:LowestorderEquivalentNonstandard2022} that the latter
already implies~\eqref{eqn:J_qo} for finite-dimensional $V_\nc$.

Let $G_\pw\in C^1(V_\pw;V^*)$ extend the semilinearity $G=G_\pw|_{V}$ 
and define the discrete nonlinear map $N_h:V_\pw\to V_\nc^*$ for any $v_\pw\in V_\pw$ by
\begin{align}\label{eqn:DSWP}
	N_h(\vpw;\varphi_\nc)\coloneqq a_\pw(\vpw,\varphi_\nc)+G_\pw(\vpw;J\varphi_\nc) - F(J\varphi_\nc)
	\qquad\text{for all }\varphi_\nc\in {V_\nc}.
\end{align}
The nonconforming discretisation of~\eqref{eqn:ASWP} seeks a discrete root 
$u_\nc\in V_\nc$ of $N_h$, namely
\begin{align}\label{eqn:ADWP}
	N_h(u_\nc; \varphi_\nc) = 0
	\qquad\text{for all }\varphi_\nc\in V_\nc.
\end{align}
If the discrete space $V_\nc$ is rich enough in that $J V_\nc\subset V$ is \emph{sufficiently} 
close to the regular root $u\in V$, the discrete problem~\eqref{eqn:ADWP} 
admits a discrete solution that is a quasi-best approximation and unique in a neighbourhood of $u$.
The term \emph{sufficiently close} relates to the smallness of the discretisation constant
\begin{align}\label{eqn:H2_H4}
	\delta_0\coloneqq\max\left\{\trbpw{(1\!-\!I)u},
	\|(1-I)A^{-1}DG_\pw(u)\|_{L(V_\nc;\Vt)}\right\}.%
\end{align}%

\begin{theorem}[quasi-best approximation] \label{thm:abstract_apriori}
	Let $u\in V$ denote a regular root of \eqref{eqn:ASWP} for $F\in V^*$ and $G_\pw\in C^1(V_\pw;V^*)$
	and assume that $DG_\pw$ is locally Lipschitz.
	For sufficiently small $\delta_0$, 
	there exist positive constants $\gamma_0,\varepsilon_0>0$
	and $\Const{qb}\geq1$ such that the discrete problem \eqref{eqn:ADWP} %
	has a unique discrete root $u_\nc \in V_\nc$ with 
	$\trbpw{u-u_\nc}<\varepsilon_0$ and
	\begin{align}\label{eqn:discrete_inf_sup}
		0<\gamma_0&\leq \inf_{0\ne v_\nc\in
		\Vh}\sup_{0\ne \varphi_\nc\in\Vh}\frac{DN_h(u_\nc;v_\nc,\varphi_\nc)}{\trbpw{v_\nc}\trbpw{\varphi_\nc}},\\
		\trbpw{ u - u_\nc} &\leq\Cqb \min_{v_\nc \in \Vh} \trbpw{u - v_\nc}.\nonumber%
	\end{align}
\end{theorem}
	\Cref{thm:abstract_apriori} is proven in~\cite{CNRS:UnifiedPrioriAnalysis2023} for quadratic semilinearities
	$G_\pw\in C^2(V_\pw;V^*)$ and holds in more generality also for locally Lipschitz $DG_\pw$;
	further details are given for completeness in \cref{apx:Proof of a priori}.
	Note that the constants $\gamma_0,\varepsilon_0$, and $\Const{qb}$ can and will depend on 
	the regular root $u\in V$, the extended semilinearity $G_\pw$, and (an upper bound of) $\|J\|$.

\begin{remark}[on the smallness assumption of $\delta_0$]\label{rem:smallness_delta_0}
	In PDE applications, 
	$\delta_0$ often relates to the maximal mesh-width $h_{\max}$ of an underlying triangulation.
	The point is that the smallness assumption on $\delta_0$ is satisfied for a sufficiently 
	fine triangulation and all subsequent refinements thereof.
	The assumption of a sufficiently small initial triangulation with $h_{\max}\lesssim \delta_0$ is natural for
	approximations of semilinear PDE, see, e.g.,
	\cite{BNRS:C0InteriorPenalty2017,CMN:PrioriPosterioriError2019,CNRS:UnifiedPrioriAnalysis2023,CGN:PosterioriErrorControl2024}
	and references therein for examples.
\end{remark}

\subsection{Existence verification of regular roots}%
\label{sub:The Newton--Kantorovich theorem}
\Cref{thm:abstract_apriori} is a prototypical a priori result that ensures the existence of 
a discrete root $u_\nc\in V_\nc$ of~\eqref{eqn:ADWP} close to a regular root $u\in V$; 
but the existence of the regular root remains an a priori assumption.
In contrast, the remaining parts of this paper will \emph{not} assume existence and
discuss a numerical proof strategy to verify existence a posteriori based on some 
(computed) $v_\nc\in V_\nc$. The discrete approximation $v_\nc$ itself can be arbitrary 
and may be obtained from another scheme with (approximate) solution $u_h\in V_h$ (e.g., 
by averaging $v_\nc= Av_h$ with $A:V_h\to V_\nc$).

The Newton--Kantorovich-type a posteriori existence verification for a regular root 
of~\eqref{eqn:ASWP} in the neighbourhood of $v_\nc\in{V_\nc}$ is based on the following condition.

\begin{enumerate}
	\myitem{\bf (N)}\label{ass:N}
	Let $v_\nc\in V_\nc$ and $\btT , \LDG ,\mtT{} >0$ satisfy 
	$2 \LDG\mtT{} <\btT ^2$ and
    \begin{itemize}
		\myitem{\bf (N1)}\label{ass:N1} $
		\|DG(v) - DG(w)\|_{L(V;V^*)}\leq \LDG \trb{v-w}$ holds for all $v,w\in V$,
		 \myitem{\bf (N2)}\label{ass:N2} $\trb{N(Jv_\nc)}_*\leq \mtT{} $,
    \myitem{\bf (N3)}\label{ass:N3} $DN(Jv_\nc)\in L(V;V^*)$ is invertible and fulfils the inf--sup condition
	\begin{align*}
		\btT \leq\inf_{0\ne v\in V}\sup_{0\ne\varphi\in V}\frac{DN(Jv_\nc;v,\varphi)}{\trb{v}\trb{\varphi}},
	\end{align*}
    \end{itemize}
\end{enumerate}
The Newton--Kantorovich theorem%
~\cite{CM:NewtonKantorovichTheorem2012,Zei:NonlinearFunctionalAnalysis1992}
has been employed in verified computations based on conforming
discretisations, see~\cite{NPW:NumericalVerificationMethods2019,LNO:ComputerassistedProofStationary2022} and references
therein.
The following version is adapted to nonconforming discretisations.
\begin{theorem}[solution existence verification]\label{thm:newton_kantorovich}
	Suppose \ref{ass:N} holds, then 
	\begin{enumerate}[label=(\alph*)]
		\item there exists a root $u\in V$ of $N$ in the neighbourhood of $v_\nc$ with the error bound
			\begin{align*}
				\trbpw{u-v_\nc}\leq \frac{\btT -\sqrt{\btT^2-2L \mtT{}}}{L }+\trbpw{(1-J)v_\nc}\eqqcolon
				\vrmT,
			\end{align*}
		\item there exists no other root $v\ne u\in V$ of $N$ such that
			\begin{align*}
				\trbpw{v-v_\nc}\leq \frac{\btT +\sqrt{\btT ^2-2L \mtT{}}}{L }-\trbpw{(1-J)v_\nc}\eqqcolon
				\vrpT,
			\end{align*}
		\item the root in (a) is a regular root with an inf--sup constant bounded from below by
			\begin{align*}
				\beta_0 \coloneqq \sqrt{\btT^2-2 L  \mtT{} }\leq \inf_{0\ne v\in V}\sup_{0\ne \varphi\in
				V}\frac{DN(u;v,\varphi)}{\trb{v}\trb{\varphi}}.
			\end{align*}
		\end{enumerate}
	\end{theorem}
\begin{table}[t]
	\centering
	\begin{tabular}{c|c|l|l|l}%
		constant&purpose& depends on& computed in&description\\\hline
		\multirow{2}{*}{$\btT$} & \multirow{2}{*}{\ref{ass:N3}} & $\bht$,
		$\kappa(\T,s)$,&\multirow{2}{*}{\cshref{thm:continuous_inf_sup}}& lower bound on the\\
					   && $\Constb{1}$, $\Constb{2}$, $\Constb{3}$ & &inf--sup constant \\ 
		$\bht$ & \eqref{eqn:bht_def} & $G_\pw$, $v_\nc$ & \cshref{lem:computation_bht}&discrete inf--sup constant\\ 
		$\Constb{1}$ & \eqref{eqn:Constb_1} & $G_\pw$, $v_\nc$ & \cshref{lem:explicit_constants_NVS}& norm of
		$DG_\pw|_{{V_\nc}}$ in $H^{-s}(\Omega)$\\ 
		$\Constb{2}$ & \eqref{eqn:Constb_2} & $G_\pw$, $v_\nc$ & \cshref{lem:explicit_constants_NVS} &norm of
		$DG_\pw|_{(1-I)V}$ in $H^{-s}(\Omega)$\\ 
		$\Constb{3}$ & \eqref{eqn:Constb_3} & $G_\pw$, $v_\nc$ & \cshref{lem:explicit_constants_NVS}& norm of
		$DG_\pw|_{(1-I)V}$ in $V^*$\\ 
		$\mtT{}$ & \ref{ass:N2} & $\operatorname{Res}_h$, $\mu_{\textrm{res}}$ &
		\cshref{thm:explicit_residual_control}&residual control $\trb{N(Jv_\nc)}_*\leq \mtT{}$\\ 
		$\kappa(\T, s)$ & \eqref{eqn:kappa_Ts} & $\|J\|$, $\kappa_\nc(\T,s)$ &
		\cshref{thm:discretisation_error}&approximation error control\\ 
		$\kappa_{\nc}(\T, s)$ & \eqref{eqn:kappaM_T_s} & --- & \cshref{lem:computation_kappaM_Ts}&discrete
		contribution to $\kappa(\T,s)$\\ 
		$\LDG$ & \ref{ass:N1} & $G$ & \cshref{lem:Lipschitz_explicit_NVS}&Lipschitz constant of $DN$ and $DG$\\ 
		$L_G$ & \eqref{eqn:mut_N2_def} & $G_\pw$, $v_\nc$ & \cshref{lem:explicit_constants_NVS}&distance of
		$G$ at $v_\nc$ and $Jv_\nc$\\ 
		$\mu_{\textrm{res}}$ & \eqref{eqn:discrete_residual} & $F$, $G_\pw$, $v_\nc$ &
		\cshref{lem:explicit_constants_NVS}&explicit residual estimator\\ 
		$\operatorname{Res}_h$ & \eqref{eqn:Resh_def} & $\|J\|$, $G_\pw$, $v_\nc$ & \cshref{lem:computation_Resh}&discrete residual\\ 
	\end{tabular}
	\caption{List of constants required for \ref{ass:N} in the existence verification theory 
		and references to their computation with application to
		the Navier--Stokes problem in \cref{sec:Algebraic_realisation_and_application}.%
	}
	\label{tab:sol_verification}%
\end{table}

\begin{proof}[Proof]
		The proof in two steps uses 
		the well-known Newton--Kantorovich theorem; see \cite{CM:NewtonKantorovichTheorem2012,
	Zei:NonlinearFunctionalAnalysis1992} for an
		overview and self-contained proof.
		Let $B_r(v)\coloneqq\{w\in V\ |\ \trb{v-w}<r\}$ denote the open ball of radius $r>0$ around $v\in V$.

		\medskip\noindent\emph{Step 1} prepares the Newton--Kantorovich theorem for the existence of a root close to
		$\ut\coloneqq Jv_\nc\in V$.
		Since $DN(\ut)$ is invertible by~\ref{ass:N3},~\cref{lem:cont_inf_sup} applies verbatim for
		\begin{align*}
			\|DN(\ut)^{-1}\|_{L(V^*;V)}\leq m\coloneqq \btT ^{-1}.
		\end{align*}
		This, the definition of the operator norm, and~\ref{ass:N2} result in
		\begin{align*}
			\trb{DN(\ut)^{-1}N(\ut)}\leq \|DN(\ut)^{-1}\|_{L(V^*;V)}\trb{N(\ut)}_*\leq
			\varepsilon\coloneqq \btT ^{-1}\mtT{} .
		\end{align*}
		Since the operator $A$ induced by the scalar product $a(\bullet,\bullet)$ is linear 
		(it is an isometry), 
		\begin{align*}
			DN(v)-DN(w) = DG(v)-DG(w)\qquad\text{for all }v,w\in V.
		\end{align*}
		Hence~\ref{ass:N1} is
		the Lipschitz continuity of $DN:V\to L(V;V^*)$ with Lipschitz constant~$L$. 

		\medskip\noindent\emph{Step 2} concludes the proof. 
		Since $Lm\varepsilon<1/2$ from step 1 and~\ref{ass:N}, the
		Newton--Kantorovich theorem provides a root $u\in B_{r_-}(\ut)\subset V$ of $N$ that is
		unique in $B_{r_+}(\ut)$ for $r_\pm=(1\pm\sqrt{1-2L m \varepsilon})/(Lm)$.
		Using $m=\btT^{-1}$ and $u\in B_{r_-}(\ut)$, (a) follows from
		\begin{align*}
			\trbpw{u-v_\nc}\leq \trb{u-\ut}+\trbpw{(1-J)v_\nc}\leq r_-+\trbpw{(1-J)v_\nc}\equiv\vrmT.
		\end{align*}
		Let $v\in V$ be any root of $N$ with $\trbpw{v-v_\nc}\leq \vrpT$.
		Since $u$ is the only root of $N$ in $B_{r_+}(\ut)$ and $\trb{v-\ut}\leq r_+$ follows from a triangle inequality, this implies $v=u$ and (b).
		The regularity in (c) of the root $u\in V$ is usually not explicitly
		stated but a consequence of the feasibility of Newton's method around $\ut$.
		Indeed, \cite[step~(iii) in the proof of Thm~3]{CM:NewtonKantorovichTheorem2012} provides 
		for all $v\in V$ with $Lm \trb{v-\ut}<1$
		the bijectivity of $DN(v)$ and
		\begin{align*}
			\|DN(v)^{-1}\|_{L(V^*;V)}\leq \frac{m}{1-Lm\trb{v-\ut}}.%
		\end{align*}
	For $u\in B_{r_{-}}(\ut)$ with $Lm\trb{u-\ut}\leq Lmr_-=1-\sqrt{1-2Lm\varepsilon}< 1$,
	the equivalence of the inf--sup constant with the inverse of $\|DN(u)^{-1}\|_{L(V^*;V)}$ from
	\cref{lem:cont_inf_sup} results in
	\begin{align*}
		\beta_0\equiv \frac{\sqrt{1-2Lm\varepsilon}}{m}=\frac{1-Lmr_-}{m}\leq \inf_{v\in V}\sup_{\varphi\in
		V}\frac{DN(u;v,\varphi)}{\trb{v}\trb{\varphi}}.
	\end{align*}
	This is (c) and concludes the proof.
	\end{proof}
	Note that the statement in~\cref{thm:newton_kantorovich}.b is empty for non-positive 
	$\vrpT\leq 0$ and guarantees local uniqueness of the root in (a) if $\vrpT>0$ is positive.
	The general theory for the verification of \ref{ass:N} requires analytical estimates, and
	the explicit constants listed in \cref{tab:sol_verification}.
	Indeed,~\ref{ass:N1} bounds the Lipschitz constant of $DG$ and is usually obtained from analytic
	arguments, while~\ref{ass:N2} requires an explicit control of the discrete residual 
	$\trb{N(Jv_\nc)}_*$ that can be obtained as described next.
	The lower bound on the discrete inf--sup constant in~\ref{ass:N3} is more involved and
	discussed in~\cref{sec:The_continuous_inf--sup_constant} below.

\subsection{Explicit residual control}%
\label{sub:Morley_reference_discretisation}
The nonconforming reference discretisation from \cref{sub:Nonconforming reference discretisation}
allows a simple explicit computation of $\mut$ required in~\ref{ass:N2}
as a postprocessing of the discrete residual
\begin{align}\label{eqn:Resh_def}
	\operatorname{Res}_h\coloneqq\sup_{0\ne \varphi_\nc\in {V_\nc}}\frac{N_h(v_\nc; \varphi_\nc)}{\trbpw{\varphi_\nc}},
\end{align}
computable from a linear solve (cf.\ \cref{lem:computation_Resh} below for details).
Let $L_G,\mu_{\textrm{res} }\geq0$ satisfy
\begin{align}\label{eqn:G_Lipschitz}
	\trb{G_\pw(v_\nc) - G(Jv_\nc)}_*
	&\leq L_G\trbpw{(1-J)v_\nc},\\
	G_\pw(v_\nc; \varphi)- F(\varphi)
	&\leq \mu_{\textrm{res}}\trb{\varphi}\label{eqn:discrete_residual}
	&&\text{for all }\varphi\in (1-JI) V.%
\end{align}%
\begin{theorem}[residual control~\ref{ass:N2}]\label{thm:explicit_residual_control}
	Given $\operatorname{Res}_h,L_G,\mu_{\textrm{res} }$ 
	from~\eqref{eqn:Resh_def}--\eqref{eqn:discrete_residual}, 
	it holds
	\begin{align}\label{eqn:mut_N2_def}
		\trb{N(Jv_\nc)}_*\leq \mtT{}\coloneqq \operatorname{Res}_h + (1+L_G)\trbpw{(1-J)v_\nc} 
		+ \|J\|\;\mu_{\textrm{res}}.
	\end{align}
\end{theorem}
\begin{proof}
	The Riesz isometry for the Riesz representation 
	$\varphi\in V$ of $N(Jv_\nc)\in V^*$ reads
	\begin{align}\label{eqn:Riesz}
		\trb{N(Jv_\nc)}_*^2 &=\trb{\varphi}^2 = N(Jv_\nc;\varphi). 
	\end{align}
	Recall $a_\pw(v_\nc, (1-I)\varphi) = 0$ and $\trbpw{I \varphi}\leq\trb{\varphi}$
	from~\eqref{eqn:I1}. This and~\eqref{eqn:DSWP} result in
	\begin{align*}
		N_h(v_\nc;I\varphi) = a_\pw(v_\nc,\varphi) +G_\pw(v_\nc, JI \varphi) - F(JI \varphi).
	\end{align*}
	Algebraic manipulations with the two displayed identities verify
	\begin{align}
		\trb{N(Jv_\nc)}_*^2
			&= N_h(v_\nc; I \varphi) - a_\pw((1-J)v_\nc,\varphi) \label{eqn:NJvM_repr}
			+ G(Jv_\nc;\varphi) - G_\pw(v_\nc;\varphi)\\
			&\quad+ G_\pw(v_\nc;(1-JI)\varphi) - F((1-JI)\varphi)\eqqcolon\mathcal{I}_1+\mathcal{I}_2.\nonumber
	\end{align}
	It remains to separately estimate the two lines, 
	denoted by $\mathcal{I}_1$ and $\mathcal{I}_2$, from the right-hand
	side of~\eqref{eqn:NJvM_repr}.
	A Cauchy inequality, $\operatorname{Res}_h$ from~\eqref{eqn:Resh_def}, and $L_G$ control 
	$\mathcal{I}_1$ by
	\begin{equation}\label{eqn:NJvM_I1}
		\begin{aligned}
			\mathcal{I}_1
			&\coloneqq N_h(v_\nc; I \varphi) -
			a_\pw((1-J)v_\nc,\varphi)
			+ G(Jv_\nc;\varphi) - G_\pw(v_\nc;\varphi)\\
			&\leq\left(\operatorname{Res}_h+(1+L_G)\trbpw{(1-J)v_\nc}\right)\trb{\varphi}.
		\end{aligned}
	\end{equation}
	The right-inverse property~\eqref{eqn:JI} of
	$J$ makes $JI\in L(V;V)$ an oblique projection, i.e., $(JI)^2=JI$, 
	in the Hilbert space $(V,a)$. Hence Kato's lemma
	\cite{Szy:ManyProofsIdentity2006} implies the identity
	\begin{align}\label{eqn:Kato_1JIM}
		\|(1-JI)\|=\|JI\|\leq \|J\|
	\end{align}
	with $\|JI\|\leq\|J\|\|I\|$ and $\|I\|\leq1$
	in the last step.
	This and $\mu_{\textrm{res}}$ from~\eqref{eqn:discrete_residual}
	bound $\mathcal{I}_2$ as
	\begin{align}\label{eqn:NJvM_I2}
		\mathcal{I}_2\coloneqq G_\pw(v_\nc;(1-JI)\varphi) - F((1-JI)\varphi)\leq
		\mu_{res}\trb{(1-JI)\varphi}\leq\|J\|\,\mu_{res}\,\trb{\varphi}.
	\end{align}
	The combination of~\eqref{eqn:NJvM_repr}--\eqref{eqn:NJvM_I1} with~\eqref{eqn:NJvM_I2} concludes the proof.
\end{proof}

\section{The continuous inf--sup constant}%
\label{sec:The_continuous_inf--sup_constant}
The lower bound for the inf--sup constant of \(DN(\ut)\) required in~\ref{ass:N3}
is obtained by adapting numerical invertibility verification
to nonconforming discretisations based on a novel approximation error control.

\subsection{Approximation error control}%
\label{sec:Explicitly computable constants}
\newcommand{\Cweak}{\kappa_{\nc}(\T,s)}

The remaining parts of this paper consider the Sobolev setting $V=H^m_0(\Omega)$
from the introduction with isometry $A=(-\Delta)^m:V\to V^*$ 
over a Lipschitz domain $\Omega\subset\mathbb{R}^n, n\geq2$.

Consider for $0\leq s\leq m$, the fractional-order Sobolev space $H^s_0(\Omega)$ 
with some norm $\|\bullet\|_s\approx\|\bullet\|_{H^s(\Omega)}$
equivalent to the standard norm
$\|\bullet\|_{H^s(\Omega)}$, and endow
its dual $H^{-s}(\Omega)$ with the corresponding operator norm
$\|\bullet\|_{s,*}$.
The nonconforming discretisation $V_\nc$
from~\cref{sub:Nonconforming reference discretisation} is associated to an
underlying admissible triangulation $\mathcal{T}$ of $\Omega$ 
with maximal mesh-size $h_{\max}$ and $\|J\|\approx 1\approx \LJ$.
Assume standard approximation properties known from classical nonconforming schemes, namely
\begin{align}\label{eqn:kappa}
	\|v\|_s&\leq \kappa_{m-s} h_{\max}^{m-s}\trb{v}
	\qquad\text{for all }v\in V \text{ with }Iv=0,\\\label{eqn:I_reg}
	\trbpw{(1-I)v}&\lesssim h_{\max}\|v\|_{H^{m+1}(\Omega)}
	\qquad\text{for all }v\in V\cap H^{m+1}(\Omega).
\end{align}
Here $0<\kappa_{m-s}$ in~\eqref{eqn:kappa} should be explicitly known, 
which motivates the flexibility in the norm $\|\bullet\|_s$ 
(cf.~\cref{rem:fractional_discretisation_error}), while
$\|\bullet\|_s=|\bullet|_{H^s(\Omega)}$ is a natural choice if $s\in\mathbb{N}_0$.

A key ingredient to the further analysis is a uniform, computable bound $\kappa(\mathcal{T},s)$
on the approximation error of $H^{-s}(\Omega)$ sources. 
More precisely, $\kappa(\mathcal{T},s)$ should satisfy
\begin{align}\label{eqn:kappa_Ts}
	\trb{z-z_\nc}_\pw\leq\kappa(\T,s)\|F\|_{s,*} 
	\qquad\text{for all }F\in H^{-s}(\Omega)
\end{align}
for $z\in V$ with $(-\Delta)^m z = F$ 
and its nonconforming approximation $z_\nc\in{V_\nc}$ with
\begin{align}\label{eqn:zM_approx_def}
	a_\pw(z_\nc, v_\nc) = F(Jv_\nc)\qquad\text{for all }v_\nc\in{V_\nc}.
\end{align}
The value $\kappa(\T,s)$ in~\eqref{eqn:kappa_Ts} is also known as a priori error estimator%
~\cite{LO:VerifiedEigenvalueEvaluation2013} and satisfies $\kappa(\mathcal{T},s)\to 0$ as $h_{\max}\to 0$ if $s<m$.
Its computation below is based on a discrete counterpart.
To this end, consider discrete solution operator 
$A_{h,s}^{-1}:{V_\nc}\to {V_\nc}$ defined by
\begin{align}\label{eqn:Ahs_def}
	a_\pw(A_{h,s}^{-1}f_\nc, \varphi_\nc) = \left\langle Jf_\nc,J\varphi_\nc\right\rangle_s\qquad\text{for all
	}f_\nc,\varphi_\nc\in{V_\nc}.
\end{align}
with scalar product $\left\langle \bullet,\bullet\right\rangle_s$ associated 
to the norm $\|\bullet\|_s\coloneqq\sqrt{\left\langle \bullet,\bullet\right\rangle_s}$.
The discrete version 
\begin{align} \label{eqn:Cdual_def}
	0<\Cweak\coloneqq \max_{0\ne g_\nc\in{V_\nc}}\min_{v_\nc\in{V_\nc}}
	\frac{\trbpw{A_{h,s}^{-1}g_\nc-Jv_\nc}}{\|Jg_\nc\|_{s}}
\end{align}
of the approximation error control~\eqref{eqn:kappa_Ts} is interpreted as a finite-dimensional eigenvalue problem in \cref{sub:computation_kappaM_Ts} and controls the approximation error  for ``discrete'' sources.
Let $\sigma_{\mathrm{reg}}>0$ denote an index of elliptic regularity of $\Omega$: 
For any $F\in H^{-m+s}(\Omega)$ with $0\leq s\leq \sigma_{\mathrm{reg}}$,
it holds $(-\Delta)^{-m}F\in H^{m+s}(\Omega)$ and 
$\|(-\Delta)^{-m}F\|_{H^{m+s}(\Omega)}\lesssim \|F\|_{H^{-m+s}(\Omega)}$.
\begin{lemma}[approximation error control for discrete sources]\label{lem:qb_discrete}
	Given any source $F\in H^{-s}(\Omega)$ with $0\leq s\leq m$, 
	consider ${z_{\mathrm{nc}}}\in {V_\nc}$
	from~\eqref{eqn:zM_approx_def} and the solution $z\in V$ to
	\begin{align}\label{eqn:z_def}
		a(z,w) &= F(J\Inc w)\qquad \text{for all }w\in V.
	\end{align}
	Then $\Cweak$ from~\eqref{eqn:Cdual_def} satisfies $\Cweak\lesssim h_{\max}^\sigma$ with
	$\sigma\coloneqq\min\{1, m-s,\sigma_{\rm reg}\}$ and
	\begin{align}\label{eqn:kappaM_T_s}
		\trbpw{z-{z_{\mathrm{nc}}}} \leq\Cweak \|F\|_{s,*}.
	\end{align}
\end{lemma}
\begin{proof}[Proof]
	The proof consists of 3 steps.

	\medskip\noindent\emph{Step 1} considers the right-hand side as a discrete function.
	The Riesz representation of $F\in H^{-s}(\Omega)$ in the Hilbert space $(H^s_0(\Omega),
	\left\langle \bullet,\bullet\right\rangle_s)$ is the unique function $f\in H^s_0(\Omega)$ with $F =
	\left\langle f,\bullet\right\rangle_s$.
	Let $G_s:H^s_0(\Omega)\to J{V_\nc}$ be the Galerkin projection onto $J{V_\nc}$, i.e.,
	\begin{align}\label{eqn:Gs_Galerkin}
		\left\langle (1-G_s)g, Jw_\nc\right\rangle_s=0\qquad\text{for all }g\in H^s_{0}(\Omega),\wM\in{V_\nc}.
	\end{align}
	Since $G_s$ is an orthogonal projection and $G_sf\in J{V_\nc}$, there exists a Morley function $f_\nc\in{V_\nc}$ with $G_sf=Jf_\nc$
	and 
	\begin{align}\label{eqn:F_norms_s}
		\|Jf_\nc\|_s=\|G_s f\|_s\leq\|f\|_s= \|F\|_{s,*}
	\end{align}
	with the Riesz isometry in the last step.
	Since $F(Jw_\nc)=\left\langle Jf_\nc, Jw_\nc\right\rangle_s$ for all $w_\nc\in {V_\nc}$ by \eqref{eqn:Gs_Galerkin},
	$z_\nc\in{V_\nc}$ from \eqref{eqn:zM_approx_def} is equal to
	$z_\nc=A_{h,s}^{-1}f_\nc$ from~\eqref{eqn:Ahs_def}, i.e.,
		\begin{align}\label{eqn:z_M_alt}
			a_\pw(z_\nc,w_\nc) &= \left\langle Jf_\nc, Jw_\nc\right\rangle_s &&\text{for all }w_\nc\in {V_\nc}.
		\end{align}

	\medskip\noindent\emph{Step 2} verifies the approximation error control.
	Since $a_\pw\left(z_\nc, (1-I)w\right)=0$ vanishes for all $w\in V$ by~\eqref{eqn:I1},
	the solution properties \eqref{eqn:zM_approx_def} of $z_\nc$ and 
	\eqref{eqn:z_def} of $z$ reveal
	\begin{align}\label{eqn:z_zM_orthogonality}
		a_\pw(z-{z_{\mathrm{nc}}}, w) = a(z, w) - a_\pw({z_{\mathrm{nc}}}, \Inc w) = 0\qquad\text{for all }w\in V.
	\end{align}
	Hence the Pythagoras theorem and $J{V_\nc}\subset V$ provide
	\begin{align*}
		\trbpw{z-{z_{\mathrm{nc}}}}=\min_{v\in V}\trbpw{{z_{\mathrm{nc}}}-v}\leq \min_{\vM\in{V_\nc}}\trbpw{{z_{\mathrm{nc}}}-J\vM}.
	\end{align*}
	This, the definition $\Cweak$ from \eqref{eqn:Cdual_def} with $z_\nc=A_{h,s}^{-1}f_\nc$ from step 1, and \eqref{eqn:F_norms_s} verify
	\begin{align*}
		\trbpw{z-{z_{\mathrm{nc}}}}\leq\Cweak\|J f_\nc\|_s\leq\Cweak\|F\|_{s,*}.
	\end{align*}

	\medskip\noindent\emph{Step 3} establishes $\Cweak\lesssim h_{\rm max}^{\sigma}$.
	Given any $g_\nc\in{V_\nc}$, let $\overline{z}\coloneqq (-\Delta)^{-m}\overline{F}\in V$ denote the exact weak solution for the right-hand side $\overline{F}\coloneqq\left\langle
	Jg_\nc,\bullet\right\rangle_s\in H^{-s}(\Omega)$.
	The elliptic regularity of $\Omega$ provides 
	$\overline{z}\in H^{m+\sigma}(\Omega)$ with $\|\overline{z}\|_{H^{m+\sigma}(\Omega)}\lesssim
	\|\overline{F}\|_{H^{\sigma-m}(\Omega)}$ so that 
	\begin{align}\label{eqn:elliptic_regularity_barz}
		\|\overline{z}\|_{H^{m+\sigma}(\Omega)}\lesssim \|\overline{F}\|_{H^{-s}(\Omega)}\approx\|\overline{F}\|_{s,*}=\|Jg_\nc\|_s
	\end{align}
	follows from a Sobolev inequality with $s\leq m-\sigma$ and the norm equivalence 
	$\|\bullet\|_{H^{s}(\Omega)}\approx \|\bullet\|_{s}$.
	Since $a_\pw\left(A_{h,s}^{-1}g_\nc, \bullet\right)=\overline{F}\circ J\in {V_\nc}^*$ 
	from~\eqref{eqn:Ahs_def}, the a priori
	error analysis~\cite[thm.~2.5]{CG:AdaptiveMorleyFEM2025}
	(also~\cite{CN:PrioriPosterioriError2021,CN:LowestorderEquivalentNonstandard2022})
	reveals $A_{h,s}^{-1}g_\nc$ as a quasi-best approximation of $\overline{z}$. Hence 
	\begin{align*}
		\trbpw{\overline{z}-A_{h,s}^{-1}g_\nc}\lesssim\trb{(1-I)\overline{z}}_\pw\lesssim 
		h_{\max}^\sigma\|\overline{z}\|_{H^{m+\sigma}(\Omega)}\lesssim h_{\max}^\sigma\|Jg_\nc\|_s
	\end{align*}
	follows with~\eqref{eqn:I_reg} and \eqref{eqn:elliptic_regularity_barz} in the last steps.
	The quasi-optimality~\eqref{eqn:J_qo} of $J$ verifies
	\begin{align*}
		\min_{\vM\in{V_\nc}}\trbpw{A_{h,s}^{-1}g_\nc-J\vM}
		\leq \trbpw{(1-J)A_{h,s}^{-1}g_\nc}\leq\LJ\trbpw{\overline{z}-A_{h,s}^{-1}g_\nc}.
	\end{align*}
	Since $g_\nc\in{V_\nc}$ was arbitrary, the combination of the two previously displayed estimates result in
	$\Cweak\lesssim h_{\rm
	max}^\sigma$ and conclude the proof.
\end{proof}

\begin{theorem}[approximation error]\label{thm:discretisation_error}
	Given $0\leq s\leq m$ and $\sigma\coloneqq\min\{1, m-s,\sigma_{\mathrm{reg}}\}$, 
	any $F\in H^{-s}(\Omega)$ satisfies~\eqref{eqn:kappa_Ts} 
	for $z\coloneqq (-\Delta)^{-m}F\in V$ 
	and $z_\nc\in {V_\nc}$ from~\eqref{eqn:zM_approx_def} with
	\begin{align}
		\kappa(\T,s)\coloneqq\sqrt{h_{\max}^{2(m-s)}\kappa_{m-s}^2\|J\|^2+\Cweak^2}\lesssim
		h_{\max}^\sigma.\label{eqn:kappa_Ts_def}
	\end{align}
\end{theorem}
\begin{proof}[Proof]
	Given any $F\in H^{-s}(\Omega)$, let $v\in V$ solve $(-\Delta)^m v=F$. 
	\Cref{lem:qb_discrete} verifies for $z_\nc\in{V_\nc}$ from~\eqref{eqn:zM_approx_def} and $z\in V$ from
	\eqref{eqn:z_def} that
	\begin{align}\label{eqn:z_zM_application}
		\trb{z-z_\nc}_\pw\leq \Cweak\|F\|_{s,*}.
	\end{align}
	The orthogonality $a_\pw(z-z_\nc, w) = 0$ from \eqref{eqn:z_zM_orthogonality} for $w\coloneqq v-z\in V$ provides
	\begin{align}\label{eqn:v_z_zM_split}
		\trb{v-z_\nc}_\pw^2=\trb{w}^2+\trb{z-z_\nc}_\pw^2.
	\end{align}	
	Since $v\in V$ is the weak solution to $(-\Delta)^m v=F$ and $z\in V$ solves \eqref{eqn:z_def},
	it holds
	\begin{align*}
		\trb{w}^2= a(v, w) - a(z, w) 
			= F\left((1-JI)w\right)
			&\leq \|F\|_{s,*}\|(1-JI)w\|_{s}.
	\end{align*}
	The right-inverse property of $J$ provides $I(1-JI)=0$.
	Hence~\eqref{eqn:kappa}
	and~\eqref{eqn:Kato_1JIM} show
	\begin{align}\label{eqn:w_s_bound}
		|(1-JI)v|_{H^s(\Omega)}\leq \kappa_{m-s}h_{\max}^{m-s}\trb{(1-JI)v}
		\leq\kappa_{m-s}h_{\max}^{m-s}\|J\|\trb{v}. %
	\end{align}%
	The combination of \eqref{eqn:z_zM_application}--\eqref{eqn:w_s_bound} reveals
	\begin{align*}
		\trbpw{v-z_\nc}^2\leq\left(\kappa_{m-s}^2h_{\max}^{2(m-s)}\|J\|^2+\Cweak^2\right)\|F\|_{s,*}^2.
	\end{align*}
	Recall $\|J\|\approx 1$.
	Since $\kappa_\nc(\T, s)\lesssim h_{\max}^\sigma$ by~\cref{lem:qb_discrete} 
	and $h_{\max}^{m-s}\lesssim h_{\max}^\sigma$ from $\sigma\leq m-s$ 
	imply $\kappa(\T,s)\lesssim h_{\max}^\sigma$ from~\eqref{eqn:kappa_Ts_def}, 
	this concludes the proof.
\end{proof}
\begin{remark}[the constant $\kappa_{m-s}$]\label{rem:fractional_discretisation_error}
	For classical nonconforming schemes such as
	(modified) Crouzeix--Raviart FEM for $m=1$ and Morley FEM for $m=2$,
	the interpolation error control~\eqref{eqn:kappa} is standard:
	Analytical~\cite{CG:GuaranteedLowerEigenvalue2014,Ma:ConstantsL2Error2021} and
	numerical~\cite{LSL:OptimalEstimationFujino2019,CG:RateoptimalHigherorderAdaptive2024} 
	bounds for $\kappa_{m-s}$ (with $\kappa_0=1$) are known if $s\in\mathbb{N}_0$
	with $s\leq m$ and $\|\bullet\|_{s}=|\bullet|_{H^s(\Omega)}$,
	and follow by interpolation if $s\not\in\mathbb{N}_0$.
	However, explicit constants in non-interpolation norms are more involved in the latter case.
	If $\|\bullet\|_s$ is the piecewise seminorm over $\mathcal{T}$,
	a natural approach is to compute $\kappa_{m-s}$ as the maximum of $\kappa_{m-s,T}$
	from the local problems
	\begin{align*}%
		|(1-I) v|_{H^s(T)}\leq 
		\kappa_{m-s,T}h_T^{m-s}\|D^mv\|_{L^2(T)}\qquad\text{for all }v\in V\text{ and }T\in\T.
	\end{align*}
	that only depend on the shape of $T\in\mathcal{T}$,
	see~\cite{LSL:OptimalEstimationFujino2019}
	and~\cite[subsec.~4.6]{CG:RateoptimalHigherorderAdaptive2024}
	for details.

	A similar localisation is not immediate 
	for the Sobolev-Slobodeckij norm $\|\bullet\|_s=|\bullet|_{H^s(\Omega)}$
	with non-integer $s\in\mathbb{R}$, where the best possible constant $\kappa_{m-s}$ 
	in~\eqref{eqn:kappa} relates to a nonlocal (eigenvalue) problem.
\end{remark}

\subsection{Guaranteed lower bound on the inf--sup constant}%
\label{sub:Lower_bound_continuous_inf--sup_constant}
The following main result of this section introduces an explicit lower bound of the 
inf--sup constant of $DN(\ut)$ at some $\ut\in V$ from a 
postprocessing of its discrete counterpart
\begin{align}\label{eqn:bht_def}
	\bht\coloneqq\inf_{0\ne w_\nc\in {V_\nc}}\sup_{0\ne \varphi_\nc\in {V_\nc}}\frac{DN_h(\ut;w_\nc,
	\varphi_\nc)}{\trbpw{w_\nc}\trbpw{\varphi_\nc}}.
\end{align}
\Cref{lem:computation_bht} below writes $\bht$ as an eigenvalue in a discrete eigenvalue problem.
Let $\|M\|_2$
denote the spectral norm of $M\in\R^{2\times 2}$, induced by the Euclidean norm $|\bullet|$ on $\R^2$.
\begin{theorem}[lower bound on continuous inf--sup constant]\label{thm:continuous_inf_sup}
	Given $\ut\in V$ and $0\leq s\leq m$ such that 
	$D\GGamma_\pw(\ut)\in L(V_\pw;H^{-s}(\Omega))$,
	let $\Constb{1}, \Constb{2}, \Constb{3}>0$ 
	satisfy
	\begin{align}\label{eqn:Constb_1}
		\|D\GGamma_{\pw}(\ut;v_\nc)\|_{s,*}
		&\leq \Constb{1}\trbpw{v_\nc}, \\%&&\text{for all }v_\nc\in{V_\nc}.
		\|D\GGamma_{\pw}(\ut;(1-I)v)\|_{s,*}
		&\leq \Constb{2}\trbpw{(1-I)v},\label{eqn:Constb_2}\\ %
		\trb{D\GGamma_{\pw}(\ut;(1-I)v)}_*
	&\leq \Constb{3}\trbpw{(1-I)v}\label{eqn:Constb_3} %
	\end{align}
	for all $v\in V$ and $v_\nc\in{V_\nc}$.
	Define $M,N\in \R^{2\times 2}$ with $\Constb{\T}\coloneqq\kappa(\T,
s)\Constb{2}+\Constb{3}\|J\|$ by
	\begin{equation}\label{eqn:MN_def}
		\begin{aligned}
			N&\!\coloneqq\! \begin{pmatrix}
				1&\Constb{\T}\\
				0&\bht
			\end{pmatrix},\qquad
			M\!\coloneqq\! \begin{pmatrix}
				\left(1+\Constb{\T}\right)\Constb{1}&\Constb{\T}\Constb{2}\\
				\bht\Constb{1}&\bht\Constb{2}
			\end{pmatrix}.
		\end{aligned}
	\end{equation}
	(a) Suppose $\kappa(\T,s)\|M\|_2<\bht$ holds, then $DN(\ut)$ satisfies the inf--sup condition
	\begin{align*}
		\btT\coloneqq\frac{\bht-\kappa(\T,s)\|M\|_2}{\|N\|_2}\leq\inf_{0\ne v\in V}\sup_{0\ne \varphi\in
		V}\frac{DN(\ut;v,\varphi)}{\trb{v}\trb{\varphi}}.
	\end{align*}
	(b) If additionally $0\leq s<m$, 
	then $DN(\ut)$ is bijective and \ref{ass:N3} holds with $\btT$ from (a).
\end{theorem}
A remark on the asymptotics of the postprocessing in~\cref{thm:continuous_inf_sup} 
precedes its proof.%
\begin{remark}[asymptotics of $\btT$]\label{rem:condition_cont_inf_sup}
	The condition in~\cref{thm:continuous_inf_sup}(a) requires a positive
	discrete inf--sup constant $0<\bht$.
	If this holds,~\cref{thm:discretisation_error} 
	guarantees $\kappa(\T,s)\|M\|_2<\bht$
	for sufficiently small mesh-sizes
	from $\kappa(\T,s)\lesssim h_{\max}^\sigma$ for
	$\sigma\coloneqq\min\{1,m-s,\sigma_{\textrm{reg}}\}$.
	Provided that $\bht$ and $\|M\|_2,\|N\|_2$ admit uniform bounds, 
	this implies the asymptotic equivalence $\btT \sim \bht/\|N\|_2$ 
	(in the sense that their ratio converges to $1$ as $h_{\max}\to 0$).
	Since $\|J\|\approx1$, 
	the convergence $\Constb{3}\to 0$ is necessary and sufficient for $\Constb{\T}\to0$,
	and in turn implies $\|N\|_2\to \max\{1,\bht\}$ and $\btT\sim\min\{1,\bht\}$.
\end{remark}

\begin{proof}
	The general idea of the proof is the separate analysis in the orthogonal spaces ${V_\nc}$ and $(1-I)V$ with
	similarities to the conforming case, e.g.,
	in~\cite{NHW:NumericalMethodVerify2005,TLO:VerifiedComputationsSemilinear2013,LNO:ComputerassistedProofStationary2022}.
	However, the present analysis in 5 steps additionally depends on the smoother $J$ and the approximation error
	control $\kappa(\T,s)$ for nonconforming functions in ${V_\nc}\not\subset V$.

	\medskip\noindent\emph{Step 1} prepares the proof of (a).
	The lower bound $\btT$ on the continuous inf--sup constant $\widehat{\beta}$ at $\ut\in V$ 
		satisfies
		\begin{align}\label{eqn:bht_inf_sup_character}
			\btT\leq\widehat{\beta}\coloneqq\inf_{0\ne v\in V}\sup_{0\ne w\in
		V}\frac{DN(\ut;v,w)}{\trb{v}\trb{w}}
		\quad\Leftrightarrow\quad\btT\trb{v}\leq\trb{DN(\ut;v)}_{*}\qquad\text{for all }v\in V.
	\end{align}
	Let $v\in V$ be arbitrary and consider the Riesz representation $z\in V$ of $DG(\ut;v)$, i.e.,
	\begin{align}\label{eqn:inf_sup_z_def}
		a(z, w) = D\GGamma(\ut;v,w)\qquad\text{for all }w\in V.
	\end{align}
	This and $DN(\ut;v) = Av + D\GGamma(\ut;v)\in V^*$ imply that 
	$\varphi\coloneqq v+z\in V$ solves
	\begin{align}\label{eqn:inf_sup_vf_def}
		a(\varphi, w) = DN(\ut;v,w) \qquad\text{for all }w\in V.
	\end{align}
	The orthogonality~\eqref{eqn:I1} of $I$ leads to the decomposition of $v$ and $\varphi$ as
	\begin{align}\label{eqn:v_vf_split}
		\trb{v}^2=\trbpw{I
		v}^2+\trbpw{(1-I)v}^2\quad\text{and}\quad\trb{\varphi}^2=\trbpw{I\varphi}^2+\trbpw{(1-I)\varphi}^2.
	\end{align}

	\noindent\emph{Step 2} controls $\trbpw{I v}$.
	Given any $v\in V$,~\eqref{eqn:bht_def} provides $w_\nc\in {V_\nc}$ with $\trbpw{w_\nc}\leq1$ and
	\begin{align*}
		\bht\trb{I v}_\pw&=DN_h(\ut;I v,w_\nc) = a_{\pw}(I v, w_\nc) + D\GGamma_{\pw}(\ut;I v,Jw_\nc),
	\end{align*}
	using the definition of $DN_h(\ut)$ in the last step.
	Let $z_\nc\in{V_\nc}$ solve
	\begin{align*}
		a_\pw(z_\nc, w_\nc) = D\GGamma(\ut; v, Jw_\nc) \qquad\text{for all }w_\nc\in {V_\nc},
	\end{align*}
	i.e., $z_\nc$ approximates the solution $z\in V$ to 
	$(-\Delta)^m z = D\GGamma(\ut;v)\equiv F\in H^{-s}(\Omega)$ by~\eqref{eqn:zM_approx_def}.
	This and $\varphi=v+z$ from step 1 combined with the inf--sup condition above reveal
	\begin{align*}
		\bht\trb{I v}_\pw
				   &=a_\pw(I \varphi, w_\nc)+a_{\pw}(z_\nc-I z, w_\nc
				   )-D\GGamma_{\pw}(\ut;(1-I)v, Jw_\nc).
	\end{align*}
	Observe $a_\pw((1-I) z, w_\nc)=0$ by~\eqref{eqn:I1} so that 
	Cauchy inequalities and the
	definition of $\Constb{3}$ result with $\trb{Jw_\nc}\leq \|J\|$ from $\trbpw{w_\nc}\leq 1$ in
	\begin{align}
		\bht\trb{I v}_\pw
				   &\leq \trbpw{I
				   \varphi}+\trbpw{z-z_\nc}+\Constb{3}\|J\|\trbpw{(1-I)v}\label{eqn:inf_sup_IM_v_bound_1}.
	\end{align}
Since $z\in V$ from \eqref{eqn:inf_sup_z_def} is the Riesz representation of
$D\GGamma(\ut;v)\in H^{-s}(\Omega)$ and $z_\nc$ its approximation in the sense of~\eqref{eqn:zM_approx_def}, the approximation error control
\eqref{eqn:kappa_Ts} verifies
\begin{align*}
	\kappa(\T, s)^{-1}\trbpw{z-z_\nc}
		&\leq \|D\GGamma(\ut;v)\|_{s,*}
	\leq\|D\GGamma_\pw(\ut;I
	v)\|_{s,*}+\|D\GGamma_\pw(\ut;(1-I)v)\|_{s,*}
\end{align*}
with a triangle inequality in the last step.
This and
the definition of $\Constb{1}$ and $\Constb{2}$ provide
\begin{align}\label{eqn:inf_sup_z_approx_bound}
	\kappa(\T, s)^{-1}\trbpw{z-z_\nc}\leq\Constb{1}\trbpw{I v}+\Constb{2}\trbpw{(1-I)v}.
\end{align}
The combination of~\eqref{eqn:inf_sup_IM_v_bound_1} with \eqref{eqn:inf_sup_z_approx_bound} results with $\Constb{\T}=\kappa(\T,
s)\Constb{2}+\Constb{3}\|J\|$ in
\begin{align}
	\bht\trb{I v}_\pw
			   \leq \trbpw{I
			   \varphi}&+\kappa(\T, s)\Constb{1}\trbpw{I v}
			   +\Constb{\T}\trbpw{(1-I)v}.\label{eqn:inf_sup_IM_v_bound_2}
\end{align}

\noindent\emph{Step 3} controls $\trbpw{(1-I)v}$.
The Galerkin property~\eqref{eqn:I1} implies $\trbpw{(1-I)z}\leq \trbpw{z-z_\nc}$.
	Hence the split $v=\varphi-z$ from \eqref{eqn:inf_sup_z_def}--\eqref{eqn:inf_sup_vf_def} and
	\eqref{eqn:inf_sup_z_approx_bound} reveal
\begin{align}\nonumber
	\trbpw{(1-I)v}
		&\leq \trbpw{(1-I)\varphi}+\trbpw{(1-I)z}\\
		&\leq \trbpw{(1-I)\varphi}+\kappa(\T,s) \left(\Constb{1}\trbpw{I
		v}+\Constb{2}\trbpw{(1-I)v}\right).\label{eqn:inf_sup_1IM_v_bound}
\end{align}
Recall 
$M=(M_{j,k})_{j,k=1,2}, N=(N_{j,k})_{j,k=1,2}\in\mathbb R^{2\times 2}$ from~\eqref{eqn:MN_def} and
abbreviate the norms $a\coloneqq\trbpw{I v}, b\coloneqq \trbpw{(1-I)v}, c\coloneqq\trbpw{I \varphi}$, and
$d\coloneqq\trbpw{(1-I)\varphi}$ to write~\eqref{eqn:inf_sup_1IM_v_bound}~as
\begin{align*}
	\bht\trbpw{(1-I)v}\leq N_{2,2}d + \kappa(\T, s)\left(M_{2,1}a + M_{2,2}b\right).
\end{align*}
Similarly, the control of $\trbpw{(1-I)v}$ in~\eqref{eqn:inf_sup_IM_v_bound_2} by \eqref{eqn:inf_sup_1IM_v_bound}
results in
\begin{align*}
	\bht\trb{I v}_\pw
			   &\leq N_{1,1}c + N_{1,2}d + \kappa(\T, s) \left(M_{1,1}a + M_{1,2}b\right).
\end{align*}

\noindent\emph{Step 4} concludes (a).
With $a,b,c$, and $d$ from step 3, the decomposition~\eqref{eqn:v_vf_split} reads
\begin{align*}
	\trb{v}=\left|\begin{pmatrix} a\\b
	\end{pmatrix}\right|,\qquad
	\trb{\varphi}=\left|\begin{pmatrix} c\\d
	\end{pmatrix}\right|.
\end{align*}
Step 3 and the monotonicity of the Euclidean norm reveal \begin{align}\label{eqn:abcd_estimate}
	\bht\trb{v} =\bht\left|\begin{pmatrix} a\\b
	\end{pmatrix}\right|
\leq \left|\kappa(\T,s)M
	\begin{pmatrix}a\\b
	\end{pmatrix}+N
	\begin{pmatrix}c\\d
	\end{pmatrix}\right|
	\leq \kappa(\T,s)\|M\|_2\trb{v}+\|N\|_2\trb{\varphi}
\end{align}
with spectral norm $\|\bullet\|_2$.
Provided $\kappa(\T,s)\|M\|_2<\bht$,
the absorption of $\kappa(\T,s)\|M\|_2\trb{v}$ in the left-hand side of \eqref{eqn:abcd_estimate} and the Riesz isometry
$\trb{\varphi}=\trb{DN(\ut;v)}_*$ result in
\begin{align*}
	\trb{v}\leq \|N\|_2/(\bht-\kappa(\T,s)\|M\|_2)\trb{DN(\ut;v)}_*.
\end{align*}
Since $v\in V$ was arbitrary, $\btT$ is a lower bound of $\widehat{\beta}$ from~\eqref{eqn:bht_inf_sup_character}.
This concludes (a).

\medskip\noindent\emph{Step 5} finalises the proof.
	It remains to show the bijectivity of $DN(\ut)$.
	If $0\leq s<m$, 
	then the Sobolev embedding $H^{-s}(\Omega)\hookrightarrow V^*$ is compact 
	and $D\GGamma(\ut)\in L(V;V^*)$ a compact linear map.
	Since $A$ is an isomorphism, 
	$DN(\ut) = A+ D\GGamma(\ut)$ is a Fredholm operator of index 0.
	Provided that $\kappa(\T, s)\|M\|_{2}<\bht$ holds, 
	step 4 verifies the injectivity (and hence bijectivity) of $DN(\ut)$.
	The relation $\|DN(\ut)^{-1}\|_{L(V^*;V)}=\widehat{\beta}^{-1}$ follows 
	from~\cref{lem:cont_inf_sup} and the lower bound $\btT$
	of $\widehat{\beta}$ from (a) concludes the proof.
\end{proof}

\subsection{Algebraic realisation}%
\label{sub:computation_kappaM_Ts}
The discrete residual $\operatorname{Res}_h$, the discrete inf--sup constant $\bht$, and
$\kappa_\nc(\T,s)$ are computable from finite-dimensional problems:
A linear solve computes the discrete residual $\operatorname{Res}_h$, 
while the discrete inf--sup constant 
$\bht$ and $\kappa_\nc(\T,s)$ for $0\leq s\leq m$ 
relate to discrete eigenvalue problems (EVP).

Let $\{\psi_j\}_{j=1}^N$ denote a basis of dimensions $\dim \Vnc = N$ for $\Vnc$ and
define the corresponding stiffness matrices $A_\nc\in\R^{N\times N}$ and $A_J\in\R^{N\times N}$ by
\begin{align}\label{eqn:AM_AJ_def}
	(A_\nc)_{j,k}&\coloneqq a_{\pw}(\psi_j, \psi_k),\qquad
	(A_J )_{j,k}\coloneqq a(J\psi_j, J\psi_k)\qquad\text{for all }j,k=1,\dots,N.
\end{align}
Recall $N_h$ and the discrete residual $\operatorname{Res}_h$ for $v_\nc\in\Vnc$ 
from~\cref{sub:Morley_reference_discretisation}.
\begin{lemma}[computation of $\operatorname{Res}_h$]\label{lem:computation_Resh}
Given $v_\nc\in\Vnc$ and $b\coloneqq (N_h(v_\nc;\psi_j))_{j=1}^{N}\in\R^N$, then $\operatorname{Res}_h=\sqrt{b^\top
	A_{\nc}^{-1}b}$.
\end{lemma}
\begin{proof}
	Given $v_\nc\in\Vnc$, the discrete residual $\operatorname{Res}_h$ is the operator norm of 
	$N_h(v_\nc)\in\Vnc^*$ with Riesz representation $z_\nc\in\Vnc$ given by
	\begin{align*}
		a_\pw(z_\nc, \varphi_\nc)=N_h(v_\nc;\varphi_\nc)\qquad\text{for all }\varphi_\nc\in\Vnc.
	\end{align*}
	This and the definition of $A_\nc$ and $b$ verify that $x\equiv(x_j)_{j=1}^N=A_\nc^{-1}b\in\R^N$ is the coefficient vector of
	$z_\nc=\sum^{N}_{j=1} x_j\psi_j$ with $\trbpw{z_\nc}^2=x^\top A_\nc x$. The Riesz isomorphism provides
	$\operatorname{Res}_h^2=\trbpw{z_\nc}^2=b^\top A_\nc^{-\top} b$ and the symmetry of $A_\nc$ concludes the proof.
\end{proof}
The right-inverse $J$ is injective and $\left\langle \bullet,\bullet\right\rangle_s$ from
\cref{sec:Explicitly computable constants} defines the 
symmetric positive-definite
Gram matrix $B_J\in\R^{N\times N}$ for the
scalar product $\left\langle J\bullet,
J\bullet\right\rangle_s$ on $\Vnc$ by 
\begin{align}\label{eqn:BJ_def}
(B_J )_{j,k}\coloneqq \left\langle J\psi_j,J\psi_k\right\rangle_s\qquad\text{for all }j,k=1,\dots,N.
\end{align}
Since $A_\nc, A_J$ from~\eqref{eqn:AM_AJ_def} and $B_J$ are symmetric matrices, so is $S\coloneqq B_J(A_\nc^{-1}-A_J^{-1})B_J$.
\begin{proposition}[computation of $\kappa_\nc(\T, s)$]\label{lem:computation_kappaM_Ts}
	The constant $\kappa_\nc(\T, s)=\sqrt{\lambda_{\max}}$ equals the square root of the maximal eigenvalue of the finite
	dimensional,
	symmetric EVP:
	\begin{flalign*}
		\quad\text{seek }(\lambda, x)\in \R\times \R^N\text{such that}\qquad S x = \lambda B_J x.&&
	\end{flalign*}
\end{proposition}
\begin{proof}
	Since $J$ is a right-inverse of $I$, $J$ is injective and
	$a(J\bullet, J\bullet)$ defines a scalar product on $\Vnc$.
	Let $g_\nc\in\Vnc$ be arbitrary and let 
	$z_\nc\in\Vnc$ solve 
	\begin{align}\label{eqn:J_min_kappa}
		a(Jz_\nc, J\varphi_\nc) = a_{\pw}(A_{h,s}^{-1}g_\nc, J\varphi_\nc)\qquad\text{for all }\varphi_\nc\in\Vnc.
	\end{align}
	Then $A_{h,s}^{-1}g_\nc -Jz_\nc\perp J\Vnc$ verifies the 
	Pythagoras identity
	\begin{align*}
		\trbpw{A_{h,s}^{-1}g_\nc-Jv_\nc}^2=\trbpw{A_{h,s}^{-1}g_\nc-Jz_\nc}^2+\trb{J(z_\nc-v_\nc)}^2\qquad\text{for all }v_\nc\in\Vnc.
	\end{align*}
	In particular, the minimum is obtained for 
	$v_\nc=z_\nc$ while the choice $v_\nc=0$ justifies
	\begin{align}\label{eqn:z_m_rep}
		\min_{v_\nc\in\Vnc}\trbpw{A_{h,s}^{-1}g_\nc-Jv_\nc}^2
		=\trbpw{A_{h,s}^{-1}g_\nc-Jz_\nc}^2 = \trbpw{A_{h,s}^{-1}g_\nc}^2-\trb{Jz_\nc}^2.
	\end{align}
	Denote the coefficient vector of $g_\nc$ by 
	$x=(x_j)_{j=1}^N\in \R^N$, i.e.,
	$g_\nc=\sum^{N}_{j=1} x_j\psi_j$ and recall the definition of the symmetric matrices 
	$A_\nc, A_J$, and $B_J$ from~\eqref{eqn:AM_AJ_def}--\eqref{eqn:BJ_def}.
	Inserting this representation of $g_\nc$ in \eqref{eqn:Ahs_def} reveals the coefficient vector $A_\nc^{-1}B_J x$
	of $A_{h,s}^{-1}g_\nc\in\R^{N}$.
	Since the right-hand side of \eqref{eqn:J_min_kappa} equals 
	$\left\langle Jg_\nc, J\varphi_\nc\right\rangle_s$ by \eqref{eqn:Ahs_def}, the
	coefficient vector of $z_\nc$ reads $A_J^{-1}B_J x\in\mathbb R^N$.
	The combination of~\eqref{eqn:Ahs_def} and \eqref{eqn:AM_AJ_def}--\eqref{eqn:J_min_kappa}
	results~in
	\begin{align*}
		\trbpw{A_{h,s}^{-1}g_\nc}^2&=a_\pw(A_{h,s}^{-1}g_\nc, A_{h,s}^{-1}g_\nc) 
			= x^\top B_JA_\nc^{-1} B_J x,\\
		\trb{Jz_\nc}^2&=a(Jz_\nc, Jz_\nc)  
			= x^\top B_JA_J^{-1} B_J x,\\
		\|Jg_\nc\|_s^2&=\left\langle Jg_\nc,Jg_\nc\right\rangle_s
			= x^\top B_J x.
	\end{align*}
	Since $g_\nc\in\Vnc$, identified with the coefficient space $\R^N$, was arbitrary,
	this, the definition of $\kappa_\nc(\T,s)$ in
	\eqref{eqn:Cdual_def}, and \eqref{eqn:z_m_rep} verify for $S\coloneqq
	B_J(A_\nc^{-1}-A_J^{-1})B_J$ that
	\begin{align*}
		\kappa_\nc(\T, s)^2 = \max_{0\ne x\in \R^N}\frac{x^\top Sx}{x^\top B_Jx}.
	\end{align*}
	Hence the Rayleigh-Ritz principle implies that $\kappa_\nc(\T,s)^2$ is the largest eigenvalue of the symmetric
	eigenvalue problem $Sx =\lambda B_Jx$. This concludes the proof.
\end{proof}
Given $\ut\in V$, the discrete inf--sup constant $\bht$ from \eqref{eqn:bht_def} relates to the bilinear form
$DN_h(\ut)\in L(\Vnc;\Vnc^*)$ represented by $\widehat{D}\in\R^{N\times N}$ with entries
\begin{align*}
	\widehat{D}_{j,k}\coloneqq DN_h(\ut;\psi_j,\psi_k)\qquad\text{for all }j,k=1,\dots,N.%
\end{align*}%
\begin{proposition}[computation of $\bht$]\label{lem:computation_bht}
	The constant $\bht=\sqrt{\lambda_{\min}}$ equals the square root of the minimal eigenvalue of the finite
	dimensional,
	 symmetric EVP:
	\begin{flalign*}
		\quad\text{seek }(\lambda, x)\in \R\times \R^N\text{such that}\qquad \widehat{D} A_\nc^{-1}\widehat{D}^\top  x =
		\lambda A_\nc x.&&
	\end{flalign*}
\end{proposition}
\begin{proof}
	Let $w_\nc\in\Vnc$ be arbitrary. 
	The unique Riesz representation $R_\nc(w_\nc)$
	of $DN_h(\ut;w_\nc)\in \Vnc^*$ in the Hilbert space $(\Vnc, a_\pw)$ satisfies
	\begin{align}\label{eqn:zM(vM)_def}
		a_\pw(R_\nc(w_\nc), \varphi_\nc) = DN_h(\ut;w_\nc, \varphi_\nc)\qquad \text{for all }\varphi_\nc\in\Vnc.
	\end{align}
	Since $\Vnc$ is finite-dimensional, the supremum and infimum in the definition of $\bht$ from \eqref{eqn:bht_def}
	are attained and the Riesz isometry verifies
	\begin{align} \label{eqn:bht_min_eig_def}
		\bht=\min_{0\ne w_\nc\in\Vnc}\frac{\trbpw{R_\nc(w_\nc)}}{\trbpw{w_\nc}}.
	\end{align}
	Let $x=(x_j)_{j=1}^N\in \R^N$ denote the coefficient vector of $w_\nc$ with $\trbpw{w_\nc}^2=x^\top A_\nc x$
	from~\eqref{eqn:AM_AJ_def}.
	The coefficient vector of $R_\nc(w_\nc)$ equals $A_\nc^{-1}\widehat{D}^\top  x\in\R^N$
	by~\eqref{eqn:zM(vM)_def} and satisfies 
	$\trbpw{R_\nc(w_\nc)}^2 = a_\pw(R_\nc(w_\nc),R_\nc(w_\nc)) 
	= x^\top \widehat{D} A_\nc^{-1} \widehat{D}^\top  x$.
	This and \eqref{eqn:bht_min_eig_def} show
	\begin{align*}
		\bht^2=\min_{0\ne x\in\R^N}\frac{x^\top \widehat{D} A_\nc^{-1} \widehat{D}^\top  x}{x^\top A_\nc x}.
	\end{align*}
	Hence the Rayleigh-Ritz principle implies that $\bht^2$ is the smallest eigenvalue of the symmetric
	eigenvalue problem $\widehat{D} A_\nc^{-1} \widehat{D}^\top x =\lambda A_\nc x$. This concludes the proof.
\end{proof}

\section{Existence verification for stationary Navier--Stokes}%
\label{sec:Algebraic_realisation_and_application}
This section applies the solution existence verification framework from
\cref{sec:Existence verification for fourth-order problems,sec:The_continuous_inf--sup_constant}
to the Morley discretisation of the Navier--Stokes stream-function formulation~\eqref{eqn:intro_NVS_SP}.

\subsection{Stream-function formulation of Navier--Stokes}%
\label{sub:Stream-function vorticity formulation of Navier--Stokes}
The stream-function formulation~\eqref{eqn:intro_NVS_SP} of the Navier--Stokes equations
is a fourth-order semilinear problem ($m=2$).
Its weak formulation~\eqref{eqn:SWP} concerns the Hilbert space $V=H^2_0(\Omega)$
and the quadratic nonlinearity $G\in C^1(V;V^*)$
that is expressed as $G(\bullet) =\Gamma(\bullet,\bullet)$ through
the bilinear map $\Gamma:V\times V\to V^*$ with
\begin{align}\label{eqn:NVS_G_def}
	\Gamma(v,w,\varphi)\coloneqq \int_\Omega\Delta v \curl w\cdot \nabla \varphi\d x\qquad\text{for all }v,w,\varphi\in
	V.
\end{align}
Local uniqueness for small loads is well known, 
but the author is not aware of previous genericity results for
regular roots and general sources $F\in V^*$.
The latter follows below from~\cref{thm:reg_roots} based
on the Sard-Smale theorem, which has already been employed for related
characterisations of the solution set~\cite{FP:ComportementGlobalSolutions1967,ST:GenericPropertiesNonlinear1979}, see also \cite[chap.~10]{Tem:NavierStokesEquations1995}
on results for loads $F$ with higher Sobolev regularity.
\begin{theorem}[regular roots]\label{thm:NVS_reg_roots}
	For any $F\in V^*$, 
	there exists at least one root $u\in V$ of \eqref{eqn:SWP} and all of
	them satisfy the a priori bound $\trb{u}\leq\trb{F}_*$. 
	If $\|\Gamma\|\trb{F}_*<1$, 
	this root is unique and regular.
	There exists an open dense subset 
	$\mathcal{O}\subset V^*$ such that for all $F\in \mathcal{O}$, 
	the number of roots of~\eqref{eqn:SWP} is finite and all of them are regular.
\end{theorem}
\begin{proof}
	The existence of solutions for all sources $F\in V^*$ with uniqueness provided $\|\Gamma\|\trb{F}_*<1$ is well
	known, see e.g.~\cite[IV.§2--3]{GR:FiniteElementMethods1986} and \cite{Tem:NavierStokesEquations1995}.
	In particular, the a priori bound follows immediately from $\trb{u}^2=F(u)$
	by the antisymmetry $\Gamma(u,u)=0$.

	\Cref{thm:reg_roots} provides the genericity of regular roots.
	Indeed, Hölder inequalities verify
	\begin{align*}
		\Gamma(v,w,\varphi)\leq\|\Delta v\|_{L^2(\Omega)}\|\nabla
		w\|_{L^4(\Omega)}\|\nabla\varphi\|_{L^4(\Omega)}\lesssim\trb{v}\trb{w}|\varphi|_{W^{1,4}(\Omega)}
		\qquad\text{for all }v,w,\varphi\in V
	\end{align*}
	with $|\curl\bullet|=|\nabla\bullet|$ and 
	the Sobolev embedding $V\hookrightarrow W^{1,4}_0(\Omega)$ in the last step.
	This and the density of $V\subset W^{1,4}_0(\Omega)$ implies that $\Gamma:V\times V\to
	\big(W^{1,4}_0(\Omega)\big)^*$ maps $V\times V$ continuously and boundedly into $\big(W^{1,4}_0(\Omega)\big)^*\subset V^*$.
	Since the latter inclusion is compact by Rellich and Schauder's theorem,
	$\Gamma:V\times V\to V^*$ is compact as a map into $V^*$.
	This compactness is also known from~\cite{WNN:CompactnessNonlinearOperator2017}, 
	therein proven in several pages.
	
	The semilinearity $G:V\to V^*$ inherits the compactness from 
	$G(v) = \Gamma(v,v)$ for all $v\in V$.
	Moreover, the a priori bound implies $\trb{v}\leq\trb{T(v)}_*$ for all $v\in V$
	and the weak coercivity of $T\coloneqq\Delta^2+G$ by~\cref{rem:weak_coercivity}.
	Hence,~\cref{thm:reg_roots} concludes the proof.
\end{proof}
The control~\ref{ass:N1} of the Lipschitz constant for $DG$ follows from explicit constants 
in the Sobolev embedding $H^2(\Omega)\hookrightarrow W^{1,4}(\Omega)$.
Numerical methods for this task 
are discussed in~\cite[sec.~7.6]{NPW:NumericalVerificationMethods2019} 
and may lead to sharper bounds.
\begin{lemma}[constant in Sobolev embedding]\label{lem:W14_sharp}
	Any $v\in V$ satisfies
	\begin{align*}%
		\|\nabla v\|_{L^4(\Omega)}\leq |\Omega|^{1/4}/\pi\trb{v}.
	\end{align*}
\end{lemma}
\begin{proof}
	The optimal constant $C=\pi^{-1}$ of the Sobolev embedding 
	$W^{1,4/3}(\R^2)\hookrightarrow L^{4}(\R^2)$ 
	is known from~\cite{Aub:ProblemesIsoperimetriquesEspaces1976,Tal:BestConstantSobolev1976}
	and a Hölder inequality reveals
	\begin{align}\label{eqn:C4_sharp}
		\|f\|_{L^4(\Omega)}\leq 1/\pi\|\nabla f\|_{L^{4/3}(\Omega)}\leq |\Omega|^{1/4}/\pi\|\nabla
		f\|_{L^2(\Omega)}\qquad\text{for all }f\in H^1_0(\Omega).
	\end{align}
	Given any $v\in V$, the definition of the $L^p$ norms and a triangle inequality provide
	\begin{align*}
		\|\nabla v\|_{L^4(\Omega)}^2
		&\equiv \|((\partial_{x_1}v)^2+(\partial_{x_2}v)^2)^{1/2}\|_{L^4(\Omega)}^2
		=\|(\partial_{x_1}v)^2+(\partial_{x_2}v)^2\|_{L^2(\Omega)}\\
		&\leq
		\|(\partial_{x_1}v)^2\|_{L^2(\Omega)} + \|(\partial_{x_2}v)^2\|_{L^2(\Omega)} \equiv
		\|\partial_{x_1}v\|_{L^4(\Omega)}^2 + \|\partial_{x_2}v\|_{L^4(\Omega)}^2.
	\end{align*}
	Since $\partial_{x_1}v,\partial_{x_2}v\in H^1_0(\Omega)$ from $v\in V\equiv H^2_0(\Omega)$, 
	this and \eqref{eqn:C4_sharp} result in
	{\reqnomode
	\begin{align*}
		\pi^2|\Omega|^{-1/2}\|\nabla v\|_{L^4(\Omega)}^2
		\leq \|\nabla \partial_{x_1}v\|_{L^{2}(\Omega)}^2+\|\nabla \partial_{x_2}v\|_{L^{2}(\Omega)}^2 =
		\trb{v}^2.\tag*{\qedhere}
	\end{align*}
	}
\end{proof}
\begin{lemma}[Lipschitz constant of $DG$]\label{lem:Lipschitz_explicit_NVS}
	\ref{ass:N1} holds for $L\coloneqq 2|\Omega|^{1/2}/\pi^{2}$.
\end{lemma}
\begin{proof}
	\Cref{lem:W14_sharp} 
	and $\|\Delta v\|_{L^2(\Omega)}=\trb{v}$
	result for any $v,\varphi,\psi\in V$ in
	\begin{align*}
		\Gamma(v,\varphi,\psi)\leq \|\Delta v\|_{L^2(\Omega)}\|\nabla \varphi\|_{L^4(\Omega)}\|\nabla \psi\|_{L^4(\Omega)}\leq
		|\Omega|^{1/2}/\pi^2\trb{v}\trb{\varphi}\trb{\psi}.
	\end{align*}
	Hence the Lipschitz constant $L=2|\Omega|^{1/2}/\pi^2$ follows from 
	{\reqnomode
	\begin{align*}
		D\GGamma(v;\varphi,\psi) - D\GGamma(w;\varphi,\psi) 
		= \Gamma(v\!-\!w,\varphi,\psi) + \Gamma(\varphi,v\!-\!w,\psi)\leq
		L\,\trb{v-w}\trb{\varphi}\trb{\psi}.\tag*{\qedhere}
	\end{align*}}
\end{proof}

\subsection{Morley discretisation for Navier--Stokes}%
\label{sub:Morley}
The nonconforming Morley finite element method utilises the ansatz space
\begin{align*}%
	\Vnc\coloneqq\left\{p\in P_{\hspace{-.13em}2}(\T)\ \middle|
		\begin{array}{lc}
			p(z)\text{ is continuous at every }z\in\mathcal V(\Omega)\text{ and } p|_{\mathcal{V}(\partial\Omega)}=0,\\
			\int_E[\nabla p]_E\cdot\nu_E\,\mathrm ds = 0 \text{ vanishes for every edge }E\in\E
	\end{array}\hspace{-.5em}\right\}
\end{align*}
The energy scalar product $a(\bullet,\bullet) = a_\pw(\bullet,\bullet)|_{V\times V}$
extends to the piecewise energy scalar product $a_\pw:V_\pw\times V_\pw\to\R$ 
in the Hilbert space $V_\pw=V+V_\nc$~\cite[rem.~2.2]{CN:PrioriPosterioriError2021} by
\begin{align*}%
	a_\pw(v_\pw, w_\pw)\coloneqq \int_\Omega D^2_\pw v_\pw:D^2_\pw w_\pw\d x
	\qquad\text{for all }v_\pw,w_\pw\in V_\pw.%
\end{align*}%
\begin{defn}[Morley interpolation
	{\cite{CG:GuaranteedLowerEigenvalue2014}}]\label{def:I}
	Given any $v_\pw\in V_\pw$, the Morley interpolation $I:V_\pw\to V_\nc$ 
	defines the degrees of freedom of $Iv_\pw\in V_\nc$ by
	\begin{align*}
		(v_\pw-I v_\pw)(z)=0 
\quad\text{ and } \quad
		\int_E\nabla(v_\pw-Iv_\pw)\cdot \nu_E\d s =0\quad\text{ for
		}z\in\mathcal{V}(\Omega),E\in\E(\Omega).
	\end{align*}
\end{defn} 
Two following lemmas (without proof) recall well-known properties of 
the Morley interpolation and a known right-inverse, 
satisfying the assumptions in~\cref{sub:Nonconforming reference discretisation}.
\begin{lemma}[properties \cite{CG:GuaranteedLowerEigenvalue2014}]\label{lem:IM}
	The Morley interpolation $I:V_\pw\to V_\nc$ satisfies~\eqref{eqn:I1}
	and~\eqref{eqn:I_reg}.
	There exist $\kappa_1\leq 0.2983$ and%
	\footnote{The estimate $\kappa_2\leq 0.2359$ from~\cite{Ma:ConstantsL2Error2021} 
	(proven in~\cite[thm.~3.35.d]{Put:DirectGuaranteedLower2024}) slightly improves the 
	analytical upper bounds in~\cite[thm.~3]{CG:GuaranteedLowerEigenvalue2014}; 
	sharp numerical
	values for $\kappa_1,\kappa_2$ are given 
	in~\cite{LSL:OptimalEstimationFujino2019,CG:RateoptimalHigherorderAdaptive2024}.}
	$\kappa_2\leq 0.2359$ such that any $v_\pw\in V_\pw$ satisfy
	\begin{align}\label{eqn:kappa_approx}
		|h_T^{\ell-2}(1-I)\,v_\pw|_{H^\ell(T)}
		\leq \kappa_{2-\ell}|(1-I)\,v_\pw|_{H^2(T)}
		\leq \kappa_{2-\ell}|v_\pw|_{H^2(T)}\quad\text{for all }\ell=0,1.
	\end{align}
\end{lemma}
Observe that~\eqref{eqn:kappa_approx} implies the approximation property~\eqref{eqn:kappa}
for $s=\ell=1$ with $\|\bullet\|_s=|\bullet|_{H^s(\Omega)}$.
The construction of a right-inverse $J\in L(V_\nc;V)$ for $I$
dates back to
\cite{Gal:MorleyFiniteElement2015,VZ:QuasioptimalNonconformingMethods2019} 
based on the conforming HCT element \cite[Chap.~6]{Cia:FiniteElementMethod2002}.
\begin{lemma}[right-inverse \cite{Gal:MorleyFiniteElement2015}]
\label{lem:MorleyCompanion}
There exists a bounded linear right-inverse $J: V_\nc\to V\cap W^{2,\infty}(\Omega)$
for $I$ (that is~\eqref{eqn:JI}) that is quasi-optimal in the sense of~\eqref{eqn:J_qo} with 
$\LJ\approx1\approx\|J\|$ exclusively depending on the shape-regularity.
\end{lemma}

To define the discrete nonlinear map $N_h:V_\pw\to V_\nc^*$ by~\eqref{eqn:DSWP},
the semilinearity $G=G_\pw|_{V}$ is extended by piecewise action of the differential operators
to $G_\pw(\bullet) \coloneqq\Gamma_{\pw}(\bullet,\bullet)\in C^1(V_\pw;V^*)$, 
expressed through $\Gamma_\pw:V_\pw\times V_\pw\to V^*$ with
\begin{align*}%
	\Gamma_{\pw}(v_\pw, w_\pw,\varphi) 
	\coloneqq\int_\Omega\left(\Delta_\pw v_\pw\right)\,
	\left(\curl_\pw w_\pw\right)\cdot\left(\nabla\varphi\right)\d x
	\qquad\text{for all }v_\pw,w_\pw\in V_\pw, \varphi\in V.
\end{align*} 
The well studied a priori and a posteriori error analysis%
~\cite{CMN:NonconformingFiniteElement2021,CNRS:UnifiedPrioriAnalysis2023,CGN:PosterioriErrorControl2024,CG:AdaptiveMorleyFEM2025} for the corresponding Morley FEM~\eqref{eqn:ADWP} 
guarantees the unique existence of a discrete root $u_\nc\in V_\nc$
of~\eqref{eqn:DSWP} in a neighbourhood of any regular solution $u\in V$ (cf.~\cref{thm:abstract_apriori}), and
quasi-optimal convergence rates of an adaptive loop.

The remaining parts of this section discuss 
the computation of the discrete quantities in \cref{tab:sol_verification}
required for the solution verification framework in 
\cref{sec:Existence verification for fourth-order problems,sec:The_continuous_inf--sup_constant}
for a general $v_\nc\in V_\nc$.
These constants relate to the Fr\'echet derivative of
$\GGamma_\pw$ with
\begin{align}\label{eqn:DGG_vt}
D\GGamma_\pw(v_\pw;w_\pw,\varphi) = \Gamma_\pw(v_\pw,w_\pw,\varphi) + \Gamma_\pw(w_\pw, v_\pw,\varphi)
\end{align}
for all $v_\pw,w_\pw\in V_\pw,\varphi\in V$.
A generalised Hölder inequality enables explicit estimates.%
\begin{lemma}[bound on $\Gamma_\pw$]\label{lem:G_gen_Hoelder}
	Given $1\leq p_1,p_2,p_3\leq \infty$
with $\frac{1}{p_1}+\frac{1}{p_2}+\frac{1}{p_3}=1$ and $t\in\R$,
\begin{align}\label{eqn:G_gen_Hoelder}
	\Gamma_\pw(v_\pw,w_\pw,\varphi)\leq\|h_\T^t\Delta_\pw v_\pw\|_{L^{p_1}(\Omega)}
	\|h_\T^{-t}\nabla_\pw w_\pw\|_{L^{p_2}(\Omega)}\|\nabla\varphi\|_{L^{p_3}(\Omega)}
\end{align}
holds for any $v_\pw,w_\pw\in V_\pw, \varphi\in V$
provided that the norms in the upper bound are finite.
\end{lemma}
\begin{proof}
	This is a generalised Hölder inequality with $|\curl\bullet|=|\nabla\bullet|$ and
	$\Gamma_\pw(\vt,\wt,\varphi)\leq\|h_\T^{t}\Delta_\pw\vt\;h_{\T}^{-t}\curl_\pw\wt\cdot\nabla\varphi\|_{L^1(\Omega)}$;
	further details are omitted.
\end{proof}
The improved Sobolev regularity $D\GGamma_\pw(Jv_\nc)\in L(\Vt;H^{-s}(\Omega))$ required in 
\cref{thm:continuous_inf_sup} is satisfied with $s=1$
and we set $\|\bullet\|_1\coloneqq|\bullet|_{H^1(\Omega)}$.
\begin{lemma}[boundedness of $D\GGamma_\pw(Jv_\nc)$]
	It holds $D\GGamma_\pw(J v_\nc)\in L(\Vt;H^{-1}(\Omega))$.
\end{lemma}
\begin{proof}
	\Cref{lem:G_gen_Hoelder} applies to $\vt=Jv_\nc\in W^{2,\infty}(\Omega)$ with 
	$\Delta Jv_\nc,|\nabla Jv_\nc|\in L^\infty(\Omega)$ and any $\wt=w_\pw\in
	\Vt\equiv V+\Vnc,\varphi\in V$ and verifies
	\begin{equation*}%
		\begin{aligned}
			\Gamma_\pw(Jv_\nc, w_\pw, \varphi)
			&\leq\|\Delta Jv_\nc\|_{L^\infty(\Omega)}\|\nabla_\pw
			w_\pw\|_{L^2(\Omega)}\|\nabla \varphi\|_{L^2(\Omega)},\\
			\Gamma_\pw(w_\pw, Jv_\nc, \varphi)
			&\leq\|\Delta_\pw w_\pw\|_{L^2(\Omega)}\|\nabla
			Jv_\nc\|_{L^\infty(\Omega)}\|\nabla \varphi\|_{L^2(\Omega)}.
		\end{aligned}
	\end{equation*}
	The discrete Sobolev inequality~\cite[lem.\ 7.6]{CNRS:UnifiedPrioriAnalysis2023} implies
	$\|\nabla_\pw w_\pw\|_{L^2(\Omega)}\lesssim\trbpw{w_\pw}$.
	This, $\|\Delta_\pw w_\pw\|_{L^2(\Omega)}\leq\sqrt2\trbpw{w_\pw}$, 
	and~\eqref{eqn:DGG_vt}~reveal
	\begin{align*}
		D\GGamma_\pw(J v_\nc;w_\pw,\varphi)
		&\lesssim\left(\|\Delta Jv_\nc\|_{L^\infty(\Omega)}+\|\nabla
Jv_\nc\|_{L^\infty(\Omega)}\right)\trbpw{w_\pw}\|\nabla \varphi\|_{L^2(\Omega)}.
	\end{align*}
	The supremum over all $w_\pw\in \Vt$ and $\varphi\in V\subset H^1_0(\Omega)$ concludes the proof with
	{\reqnomode
	\begin{align*}
		\|D\GGamma_\pw(J v_\nc)\|_{L(\Vt;H^{-1}(\Omega))}\lesssim 
		\|\Delta Jv_\nc\|_{L^\infty(\Omega)}\!+\!\|\nabla
		Jv_\nc\|_{L^\infty(\Omega)}\lesssim\|Jv_\nc\|_{W^{2,\infty}(\Omega)}<\infty.\tag*{\qedhere}
	\end{align*}%
}
\end{proof}

\subsection{Explicit constants}%
\label{sub:Explicit constants}

Recall $\kappa_1\leq 0.2983,\kappa_2\leq 0.2359$ from~\cref{lem:IM} and 
define the triangulation parameters
\begin{align*}
	\Const{tr,1}\coloneqq\max_{T\in\mathcal{T}}\max_{x\in T}|x-\operatorname{mid}(T)|/h_T<2/3
	\qquad\text{and}\qquad
	\mathfrak{h}(E)\coloneqq3|E|/\left(h_{T_+}^{-4}|T_+|+ h_{T_-}^{-4}|T_-|\right)
\end{align*}
for any interior edge $E=\partial T_+\cap \partial T_-\in\mathcal{E}(\Omega)$ 
shared by $T_\pm\in\T$.
The $L^2(\Omega)$ source $F\equiv (f,\bullet)_{L^2(\Omega)}$ 
in~\eqref{eqn:intro_NVS_SP} and its weak formulation~\eqref{eqn:SWP} is represented by
$f\in L^2(\Omega)$ in
\begin{align*}
	\mu_{\textrm{res}}(\T)\coloneqq \kappa_2\|h_\T^2 f\|_{L^2(\Omega)} + \sqrt{\kappa_2^2+
	\Const{tr,1}\kappa_1\kappa_2}\sqrt{\sum^{}_{E\in\E(\Omega)} \mathfrak{h}(E)\|\jump{\Delta v_\nc \nabla v_\nc}\cdot\tau_E\|_{L^2(E)}^2}.
\end{align*}
(General sources $F\in V^*$ can be treated with the techniques from~\cite[sec.\ 7]{CGN:UnifyingPosterioriError2024}.)
With Poincar\'e constant $\Const{P}\leq 1/\pi$, consider for any $v_\nc\in V_\nc$ and $w\in V$
the abbreviations
\begin{align}\label{eqn:gp_def}
	g_\infty(v_\nc)&\coloneqq{\sqrt2\|\nabla J v_\nc\|_{L^\infty(\Omega)} 
		+ \kappa_1\|h_\T\Delta J v_\nc\|_{L^\infty(\Omega)}},\\
	\Const{G}(w)&\coloneqq 
	\sqrt{2}\bigg(\frac{|\Omega|^{1/4}}{\pi}\|(1-\Pi_0)\nabla w\|_{L^4(\Omega)}+\Const{P}
\|h_\T \Pi_0\nabla w \|_{L^\infty(\Omega)}\bigg).\label{eqn:CG_def}
\end{align}
Finally, define the bilinear map $\gamma:\Vnc\times\Vnc\to L^2(\Omega;\mathbb{R}^2)$ 
for all $w_\nc,\varphi_\nc\in\Vnc$ by
\begin{align}\label{eqn:gamma_form_def}
	\gamma(w_\nc, \varphi_\nc)\coloneqq\Delta Jw_\nc\nabla_\pw \varphi_\nc + \Delta_\pw \varphi_\nc \nabla Jw_\nc\in
	L^2(\Omega;\mathbb{R}^2).
\end{align}
Recall the constants $L_G, \mu_{\textrm{res}}, \Constb{1}, \Constb{2}, \Constb{3}$ from
\cref{thm:explicit_residual_control,thm:continuous_inf_sup}.
\begin{lemma}[explicit constants for Navier--Stokes]\label{lem:explicit_constants_NVS}
	Given $v_\nc\in\Vnc$ from~\ref{ass:N}, then
	\begin{lemlist}
		\item the estimates~\eqref{eqn:G_Lipschitz}--\eqref{eqn:discrete_residual} in
			\cref{sub:Morley_reference_discretisation} hold for $F\equiv f\in L^2(\Omega)$
			in~\eqref{eqn:SWP} with
			\begin{align*}
				L_G\coloneqq \Const{G}(Jv_\nc) + \kappa_1\frac{|\Omega|^{1/4}}{\pi} \|h_\T\Delta_\pw v_\nc\|_{L^4(\Omega)},\quad \mu_{\textrm{res}}\coloneqq \mu_{\textrm{res}}(\T),
			\end{align*}
		\item the estimates~\eqref{eqn:Constb_1}--\eqref{eqn:Constb_3} in~\cref{thm:continuous_inf_sup} hold with
			\begin{align*}
				\Constb{1}\coloneqq \lambda_{\min}^{-1/2},
				\quad \Constb{2}\coloneqq g_\infty(v_\nc),\quad\Constb{3}\coloneqq \Const{G}(Jv_\nc) + \kappa_1\frac{|\Omega|^{1/4}}{\pi} \|h_\T\Delta Jv_\nc\|_{L^4(\Omega)},
			\end{align*}
	\end{lemlist}
	where 
	$\lambda_{\min}$ is the smallest eigenvalue of the discrete EVP that seeks $(\lambda,z_\nc)\in \R\times \Vnc:$ 
	\begin{align}\label{eqn:DG_JvM_wM_EVP}
		a_\pw(z_\nc, \varphi_\nc) = \lambda\, b(z_\nc, \varphi_\nc) \qquad\text{for all }\varphi_\nc\in\Vnc
	\end{align}
	with semi-scalar product $b(\bullet,\bullet)\coloneqq (\gamma(v_\nc,\bullet), \gamma(v_\nc,\bullet))_{L^2(\Omega)}:\Vnc\times\Vnc\to \R$.
\end{lemma}
The following lemma enables the convergence $\Constb{3}\to 0$ that facilitates 
optimal asymptotics $\btT\sim\bht$, 
see \cref{rem:condition_cont_inf_sup} for details,
and precedes the proof of
\cref{lem:explicit_constants_NVS} below.%
\begin{lemma}\label{lem:sharp_Gpw_2_bound}
	Any $v_\pw\in \widehat V\equiv V+\Vnc$ with $I v_\pw=0$ and $w,\varphi\in V$ satisfy
	\begin{align*}%
		\Gamma_\pw(v_\pw, w, \varphi) 
		&\leq \Const{G}(w)\trbpw{v_\pw}\trb{\varphi}.
	\end{align*}
\end{lemma}
\begin{proof}
	Algebraic manipulations reveal for any $v_\pw\in V+\Vnc$ and $w,\varphi\in V$
	as in the proof of~\cite[lem.~8.9.c]{CNRS:UnifiedPrioriAnalysis2023} the decomposition
	\begin{align}\label{eqn:Gpw_vpw_w_vf_split}
		&\Gamma_\pw(v_\pw, w, \varphi) 
		= \int_\Omega \Delta_\pw v_\pw \left(( (1-\Pi_0)\curl w)\cdot \nabla \varphi\right)\d x \\
		&+\int_\Omega \Delta_\pw v_\pw \left((\Pi_0\curl w)\cdot ((1-\Pi_0)\nabla \varphi)\right)\d x
		+\int_\Omega \Delta_\pw v_\pw \left((\Pi_0\curl w)\cdot (\Pi_0\nabla \varphi)\right)\d x.\nonumber
	\end{align}
	Since $I v_\pw$ interpolates the integral mean
	of the normal derivative of $v_\pw\in V+\Vnc$ over an edge $E\in
	\E$ exactly,
	a piecewise integration by parts verifies for any $p_0\in P_{0}(\T)$ that
	\begin{align*}
		(\Delta_\pw (1-I) v_\pw, p_0)_{L^2(\Omega)}  =0.
	\end{align*}
	Hence the assumption $I v_\pw = 0$ results for $p_0\coloneqq (\Pi_0\curl w)\cdot (\Pi_0\nabla \varphi)\in
	P_0(\T)$ in
	\begin{align}\label{eqn:Gpw_vpw_w_vf_split3bound}
		\int_\Omega \Delta_\pw v_\pw \left((\Pi_0\curl w)\cdot (\Pi_0\nabla \varphi)\right)\d x = (\Delta_\pw I v_\pw, p_0)_{L^2(\Omega)} = 0.
	\end{align}
	A Hölder inequality and $|\nabla\bullet|=|\curl\bullet|$ 
	imply
	\begin{align*}
		\|((1-\Pi_0)\curl w)\cdot \nabla \varphi\|_{L^2(\Omega)}
					 &\leq\|(1-\Pi_0)\nabla w\|_{L^4(\Omega)}\|\nabla\varphi\|_{L^4(\Omega)}\\
					 &\leq
					 \frac{|\Omega|^{1/4}}{\pi}\|(1-\Pi_0)\nabla w\|_{L^4(\Omega)}\trb{\varphi},\\
		\|(\Pi_0\curl w)\cdot ((1-\Pi_0)\nabla \varphi)\|_{L^2(\Omega)}
					 &\leq\|h_\T\Pi_0\nabla
		w\|_{L^\infty(\Omega)}\|h_\T^{-1}(1-\Pi_0)\nabla\varphi\|_{L^2(\Omega)}\\
					 &\leq \Const{P}\|h_\T\Pi_0\nabla
		w\|_{L^\infty(\Omega)}\trb{\varphi}
	\end{align*}
	with the explicit Sobolev embedding, \cref{lem:W14_sharp}, and
	the Poincar\'e inequality in the respective last steps.
	A Cauchy inequality and $\|\Delta_\pw v_\pw\|_{L^2(\Omega)}\leq \sqrt{2}\trbpw{v_\pw}$
	reveal
	\begin{align*}
		\int_\Omega \Delta_\pw v_\pw
		\left(( (1-\Pi_0)\curl w)\cdot \nabla \varphi\right)\d x
			&\leq \sqrt2\frac{|\Omega|^{1/4}}{\pi}\trbpw{v_\pw}
			\|(1-\Pi_0)\nabla w\|_{L^4(\Omega)}\trb{\varphi},\\
		\int_\Omega \Delta_\pw v_\pw 
		\left(( \Pi_0\curl w)\cdot (1-\Pi_0)\nabla \varphi\right)\d x
			&\leq\sqrt2\Const{P}\trbpw{v_\pw}\|h_\T\Pi_0\nabla
			w\|_{L^\infty(\Omega)}\trb{\varphi}.
	\end{align*}
	This and~\eqref{eqn:Gpw_vpw_w_vf_split}--\eqref{eqn:Gpw_vpw_w_vf_split3bound} conclude the proof of
	\cref{lem:sharp_Gpw_2_bound} with $\Const{G}(w)$
	from~\eqref{eqn:CG_def}.
\end{proof}

\begin{proof}[Proof of~\cref{lem:explicit_constants_NVS}]

	\textbf{ad \eqref{eqn:G_Lipschitz}}: Algebraic manipulations with $\GGamma_\pw
	=\Gamma_\pw(\bullet,\bullet)$
	show
	\begin{align}\label{eqn:GvM_JvM_split}
		G_\pw(v_\nc) -G(Jv_\nc) = \Gamma_\pw(v_\nc, (1-J)v_\nc) +\Gamma_\pw((1-J)v_\nc, Jv_\nc).
	\end{align}
	\Cref{lem:G_gen_Hoelder} with exponents $p_1=2p_2=p_3=4$ and $t=1$ verifies
	\begin{align*}
		\Gamma_\pw(v_\nc, (1-J)v_\nc,\varphi)
		&\leq \|h_\T
		\Delta_\pw v_\nc\|_{L^4(\Omega)}\|h_{\T}^{-1}\nabla_\pw(1-J)v_\nc\|_{L^2(\Omega)}\|\nabla \varphi\|_{L^4(\Omega)}
	\end{align*}
	for any $\varphi\in V$.
	Since $I v_\pw=0$ for $v_\pw\coloneqq(1-J)v_\nc$ by~\eqref{eqn:JI},
	\cref{lem:IM} plus~\cref{lem:W14_sharp} in the previous estimate 
	and \cref{lem:sharp_Gpw_2_bound} establish
	\begin{align}\label{eqn:GvM_JvM_bound1}
		\trb{\Gamma_\pw(v_\nc, (1-J)v_\nc)}_*
		&\leq \kappa_1|\Omega|^{1/4}/\pi\|h_\T
		\Delta_\pw v_\nc\|_{L^4(\Omega)}\trbpw{(1-J) v_\nc},\\
		\label{eqn:GvM_JvM_bound2}
		\trb{\Gamma_\pw((1-J)v_\nc,Jv_\nc)}_*
		&\leq \Const{G}(Jv_\nc)\trbpw{(1-J)v_\nc}.
	\end{align}
	Finally, \eqref{eqn:GvM_JvM_split}--\eqref{eqn:GvM_JvM_bound2} verify
	\eqref{eqn:G_Lipschitz} for $L_G\coloneqq\Const{G}(Jv_\nc) + \kappa_1|\Omega|^{1/4}/\pi \|h_\T
			\Delta_\pw v_\nc\|_{L^4(\Omega)}$.

\medskip\noindent\textbf{ad \eqref{eqn:discrete_residual}}: Given $w=(1-JI) v\in V$ for $v\in V$, $I w=0$ vanishes from
the right-inverse property~\eqref{eqn:JI} of $J$.
A Cauchy inequality and~\cref{lem:IM} for $\ell=0$ reveal
	\begin{align*}
		|F(w)|=\left|\int_{\Omega}
		f\,w\d
		x\right|\leq\|h_\T^2f\|_{L^2(\Omega)}\|h_\T^{-2}w\|_{L^2(\Omega)}\leq \kappa_2\|h_\T^2f\|_{L^2(\Omega)}\trb{w}.
	\end{align*}
	Standard arguments with a piecewise integration by parts and a Cauchy inequality in
	$\ell^2$ with the modified mesh size $\mathfrak{h}(E)$ verify 
	(as in the proof of~\cite[lem.~4.2]{CG:AdaptiveMorleyFEM2025}) that
	\begin{align}\nonumber
		\Gamma_\pw(\uM,\uM,w) 
			&\leq \sqrt{\sum^{}_{E\in\E(\Omega)} \mathfrak{h}(E)\|[\Delta\uM\nabla\uM]_E\cdot
			\tau_E\|_{L^2(E)}^2}\sqrt{\sum^{}_{E\in\E(\Omega)}
		\mathfrak{h}(E)^{-1}\|w\|_{L^2(E)}^2}
	\end{align}
	A straightforward generalisation of the explicit global trace 
	inequality~\cite[lem.~3]{BCGT:StabilizationfreeHHOPosteriori2023} 
	and~\cref{lem:IM} (using $\Inc w=0$) control the last term by
	\begin{align*}
		\sum^{}_{E\in\E(\Omega)} \mathfrak{h}(E)^{-1}\big\Vert w\big\Vert_{L^2(E)}^{2} 
		&\leq \big\Vert h_\T^{-2}w\big\Vert^{2}_{L^2(\Omega)} + 
		\Const{tr,1}\big\Vert h_\T^{-2}w\big\Vert_{L^2(\Omega)}
		\|h_{\T}^{-1}\nabla w\|_{L^2(\Omega)}\\
		&\leq (\kappa_2^2+\Const{tr,1}\kappa_1\kappa_2)\trb{w}^2.
	\end{align*}
	The conclusion of the previously displayed estimates and a Cauchy inequality reads
	\begin{align*}
		|\Gamma_\pw(v_\nc,v_\nc, w) - F(w)|\leq \mu_{\textrm{res}}(\T)\trb{w}.
	\end{align*}
	This and $G_\pw(v_\nc;w) = \Gamma_\pw(v_\nc, v_\nc, w)$ conclude the proof of \eqref{eqn:discrete_residual}.

\medskip\noindent\textbf{ad \eqref{eqn:Constb_1}}: A Cauchy inequality and 
the definition of $\gamma:\Vnc\to L^2(\Omega)$ in~\eqref{eqn:gamma_form_def} verify 
	\begin{align}\label{eqn:DG_JvM_wM_bound}
		DG_\pw(Jv_\nc; w_\nc, \varphi)
		\leq \|
		\gamma(v_\nc,w_\nc)\|_{L^2(\Omega)}\|\nabla\varphi\|_{L^2(\Omega)}\qquad\text{for all }\varphi\in V.
	\end{align}
	The continuous bilinear map $b(\bullet,\bullet)$
	is symmetric and positive semi-definite (while
	$a_\pw(\bullet,\bullet)$ is a scalar product for $\Vnc$). Indeed, any $w_\nc\in\Vnc$ satisfies
	$b(w_\nc, w_\nc) = \|\gamma(v_\nc, w_\nc)\|_{L^2(\Omega)}^2$.
	Hence the min-max principle
	characterises the minimal eigenvalue $\lambda_{\min}$ of~\eqref{eqn:DG_JvM_wM_EVP} as
	\begin{align*}
		\lambda_{\min}=\min_{w_\nc\in\Vnc}\frac{a_\pw(w_\nc, w_\nc)}{b(w_\nc, w_\nc)}\equiv \min_{w_\nc\in\Vnc}\frac{\trbpw{w_\nc}^2}{\|\gamma(v_\nc,w_\nc)\|_{L^2(\Omega)}^2}.
	\end{align*}
	Since $\|\bullet\|_1=\|\nabla\bullet\|_{L^2(\Omega)}$, the supremum over all $\varphi\in V$, dense in
	$H^1_0(\Omega)$,
	in~\eqref{eqn:DG_JvM_wM_bound} and the previous bound on $\|\gamma(v_\nc,w_\nc)\|_{L^2(\Omega)}$
	verify~\eqref{eqn:Constb_1}
	for $\Constb{1}\coloneqq \lambda_{\min}^{-1/2}$.

	\medskip\noindent\textbf{ad \eqref{eqn:Constb_2}}:
	\Cref{lem:G_gen_Hoelder} with adjusted exponents $p_1=\infty$, $p_2=p_3=2$ and $t=1$ (resp.~$p_2=\infty$,
	$p_1=p_3=2$ and $t=0$) verify for any $\varphi\in V$ that
	\begin{align*}
		\Gamma_\pw(Jv_\nc, (1-I)v,\varphi)
		&\leq \|h_\T
		\Delta Jv_\nc\|_{L^\infty(\Omega)}\|h_{\T}^{-1}\nabla_\pw(1-I)v\|_{L^2(\Omega)}\|\nabla
		\varphi\|_{L^2(\Omega)},\\
		\Gamma_\pw((1-I)v,Jv_\nc,\varphi)
		&\leq \|\Delta_\pw (1-I)v\|_{L^2(\Omega)}
		\|\nabla Jv_\nc\|_{L^\infty(\Omega)}\|\nabla \varphi\|_{L^2(\Omega)}.
	\end{align*}
	This, 
	\cref{lem:IM}, and $\|\Delta_\pw \bullet\|_{L^2(\Omega)}\leq\sqrt 2\trbpw{\bullet}$
	result with~\eqref{eqn:DGG_vt} in
	\begin{align*}
		D\GGamma_\pw
		&(Jv_\nc; (1-I)v,\varphi) =\Gamma_\pw(Jv_\nc, (1-I)v,\varphi) + 
		\Gamma_\pw((1-I)v,Jv_\nc,\varphi)\\
				  &\leq\left(\sqrt{2}\|\nabla Jv_\nc\|_{L^\infty(\Omega)}+\kappa_1\|h_\T
	\Delta Jv_\nc\|_{L^\infty(\Omega)}\right)\trbpw{(1-I)v}\| \nabla\varphi\|_{L^2(\Omega)}.
	\end{align*}
	The first term on the right-hand side equals $g_\infty(v_\nc)$ from~\eqref{eqn:gp_def} and
	the supremum 
	over all $\varphi\in V$
	provides~\eqref{eqn:Constb_2} for $\Constb{2}\coloneqq g_\infty(v_\nc)$ 
	by density of $V\subset H^1_0(\Omega)$.

	\medskip\noindent\textbf{ad \eqref{eqn:Constb_3}}:
	Let $v\in V$ be arbitrary and apply~\cref{lem:sharp_Gpw_2_bound} for $v_\pw=(1-I) v$ 
	and $w=Jv_\nc$.
	The arguments for~\eqref{eqn:GvM_JvM_bound1}--\eqref{eqn:GvM_JvM_bound2} in the proof
	of~\eqref{eqn:G_Lipschitz} above with $v_\nc$ and $(1-J)v_\nc$ 
	replaced by $Jv_\nc$ and $(1-I)v$ provide
	\begin{align*}
			\trb{\Gamma_\pw(Jv_\nc, (1-I)v)}_*
			&\leq \kappa_1\frac{|\Omega|^{1/4}}{\pi}\|h_\T
			\Delta_\pw Jv_\nc\|_{L^4(\Omega)}\trbpw{(1-I)v},\\
			\trb{\Gamma_\pw((1-I)v,Jv_\nc)}_*
		&\leq \Const{G}(Jv_\nc) \trbpw{(1-I)v}.
	\end{align*}
	This and the split
	\begin{align*}
		D\GGamma_\pw
		&(Jv_\nc, (1-I)v) = \Gamma_\pw(Jv_\nc, (1-I)v) + 
		\Gamma_\pw((1-I)v,Jv_\nc)%
	\end{align*}
	from~\eqref{eqn:DGG_vt} provide~\eqref{eqn:Constb_3} with $\Constb{3}\coloneqq\Const{G}(J v_\nc) +
	\kappa_1|\Omega|^{1/4}/\pi \|h_\T
			\Delta_\pw Jv_\nc\|_{L^4(\Omega)}$.
\end{proof}

\section{Numerical experiments}%
\label{sub:Numerical experiments}

Numerical computations for the solution existence verification 
with the rate-optimal Morley FEM on adaptive meshes conclude this paper.

\subsection{Numerical realisation}%
	\label{sub:Numerical realisation}
	The implementation\footnote{The source code and experiment scripts are available under
	\url{https://gitlab.com/afem}.} 
	is based on the modular Julia framework \textit{AFEM.jl}~\cite{Gra:AFEMjlV012024}.
	\subsubsection*{Integration and solvers}%
	\label{ssub:Itegration and direct solver}
	The quasi-optimal smoother $J$ is a straightforward averaging of
	Hsieh-Clough-Tocher (HCT) degrees of freedom.
	The integration of polynomial expressions is carried out exactly.
	Errors in approximating non-polynomial expressions by polynomials of sufficiently high 
	degree are expected to be very small and are neglected for simplicity.
	Linear algebraic problems are solved directly with the \texttt{mldivide} operation in Julia.
	Generalised eigenvalue problems $Ax = \lambda Bx$ with symmetric $B=C^\top C$ were
	transformed into the standard form $C^{-\top}AC^{-1}x = \lambda x$ for a noticeable speed-up in the
	computations and iteratively solved with the Krylov-Schur algorithm \texttt{eigsolve}
	from~\cite{Hae:KrylovKitV0712024};
	the tolerance was set to \texttt{tol=1e-2} for the computation of $\|J\|$ and \texttt{tol=1e-10} otherwise.
	
	The nonlinear algebraic system~\eqref{eqn:ADWP} is solved for $u_\nc\in V_\nc$
	with nested Newton iterations from initial guess $u_0=0$ and the two-level stopping
	criterion described in~\cite[SM.~3]{CGN:PosterioriErrorControl2024} to near machine precision,
	and the verification framework is applied to $v_\nc\coloneqq u_\nc$.

	\subsubsection*{Domains and refinement}
	The experiments discuss three different refinement strategies starting from the initial 
	triangulation $\T_0$ of the unit square and L-shaped domain displayed 
	in \cref{fig:initial_mesh}.
	\begin{itemize}[leftmargin=*]
		\item \emph{Uniform red refinement} into four subtriangles by connecting the edge midpoints.
		\item \emph{Adaptive refinement} with the standard adaptive loop 
			(e.g.,~\cite[fig.~1]{CG:AdaptiveMorleyFEM2025}) 
			driven~by 
		\begin{align*}
			\eta^2(T)&\coloneqq|T|^2\|f\|_{L^2(T)}^2
			+|T|^{1/2}\hspace{-.3em}\sum^{}_{E\in\E(T)}\hspace{-.3em}\Big(
			|T|\|\jump{\Delta v_\nc \nabla v_\nc}\cdot \tau_E\|_{L^2(E\setminus\partial\Omega)}^2+\|[D^2\uM]_E\tau_E\|_{L^2(E)}^2\Big)
		\end{align*}
		for a triangle $T\in\mathcal{T}$, that
			marks and refines (using NVB) all edges 
			of a minimal set of triangles $\mathcal{M}\subset\T$ satisfying the
			Dörfler bulk criterion 
			$\tfrac{1}{2}\sum^{}_{T\in\T} \eta^2(T)\leq \sum^{}_{T\in\mathcal{M}} \eta^2(T).$
		\item \emph{Adaptive refinement with mesh size reduction} extending the standard adaptive loop by additionally
			marking (only) the refinement edge of all triangles $T\in\T$ with $h_T=h_{\max}$.
	\end{itemize}
	The third refinement strategy simplifies the modified 
	marking strategy of~\cite[prop.\ 16]{BHP:AdaptiveFEMCoarse2017}
	and ensures convergence of the maximal mesh size $h_{\max}\to 0$:
	The initial triangulations into equisized 
	right-isosceles triangles with hypotenuse selected as the refinement edge
	lead to a reduction of $h_{\max}$ by a factor of $2^{-1/2}$ in each iteration.
	All computed triangulations consist of right-isosceles triangles that allow for 
	the improved constants
	\begin{align*}
		\Const{P}=\frac{1}{\sqrt{2}\pi},\quad\Const{tr,1}=\frac{\sqrt5}{3\sqrt2},\quad\kappa_1= 0.1653,\quad \kappa_2=0.0451,
	\end{align*}
	see~\cite{LSL:OptimalEstimationFujino2019}, 
	\cite[ex.\ 1]{BCGT:StabilizationfreeHHOPosteriori2023}, 
	and~\cite[rem.\ 4.2]{CG:RateoptimalHigherorderAdaptive2024} for references.
	Ndof abbreviates the number of degrees of freedom and the global error estimator reads
	$\eta\coloneqq\sqrt{\sum^{}_{T\in\mathcal{T}} \eta^2(T)}$.
\begin{figure}[!hb]
	\centering
	\scalebox{.9}{\hbox{%
			\includegraphics[]{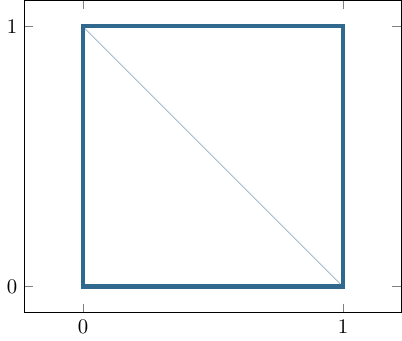}
			\includegraphics[]{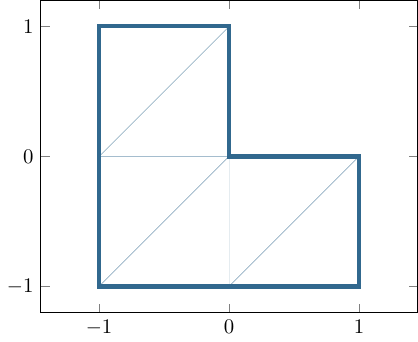}
	}}
	\caption{Initial triangulation $\T_0$ for the unit square and L-shaped domain.}%
	\label{fig:initial_mesh}
\end{figure}%

	\subsection{Academic example on the unit square}%
	\label{sec:Academic_benchmark_infsup}
	This benchmark considers the source $F_{\lambda}\equiv f_{\lambda}\in L^2(\Omega)$ matching
	the smooth solution
	\begin{align*}
		u(x,y)\equiv u_{\lambda}(x,y) = \lambda\, x^2(1-x)^2y^2(1-y)^2\in H^2_0(\Omega)
	\end{align*}
	for some parameter $\lambda\in\{1,10,100\}$
	on the unit square $\Omega=(0,1)^2$ with initial triangulation displayed in \cref{fig:initial_mesh}.
	As shown in \cref{fig:Square_Morley_Error_Convergence},
	all three refinement strategies from \cref{sub:Numerical realisation}
	lead to the optimal convergence rate $0.5$ of
	the exact error $\trbpw{u-u_\nc}$, 
	the error estimator $\eta(\mathcal{T})$, and 
	(on sufficiently fine meshes) the guaranteed error bound $\vrmT$ from
	\cref{thm:newton_kantorovich}.
	The latter has (undisplayed) efficiency
	indices $EF(\vrmT)=\vrmT/\trbpw{u-u_\nc}$ between $8$ and $10$ on fine meshes.
	\begin{figure}[!h]
		\centering
		\includegraphics{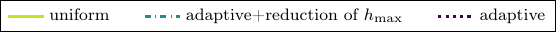}\\
		\includegraphics{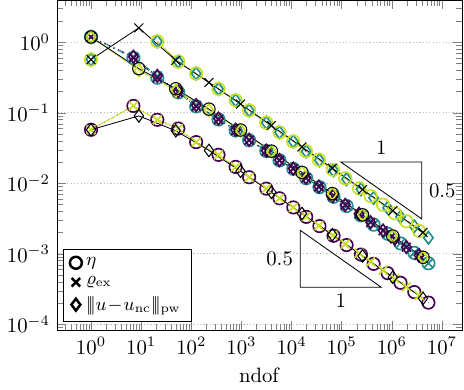}
		\includegraphics{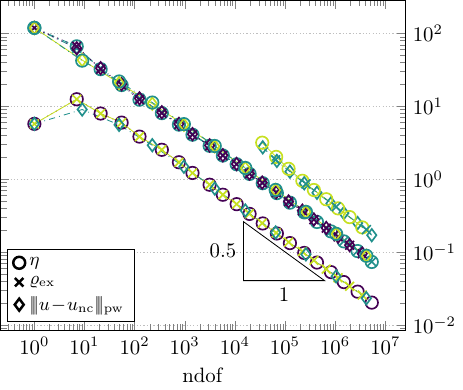}
		\caption{Convergence history plots of the error $\trbpw{u-u_\nc}$,
			the error estimator $\eta$, and the
			guaranteed error bound $\vrmT$ from \cref{thm:newton_kantorovich}.a
			under the three refinement strategies
			$\lambda=1$ (left) and $\lambda=100$ (right) in \cref{sec:Academic_benchmark_infsup}.
		 }
		\label{fig:Square_Morley_Error_Convergence}
	\end{figure}

	\begin{figure}[!h]
		\centering
		\includegraphics{./Figures/NVS/Legend_uniform_adaptive_color_out.pdf}\\
		\includegraphics{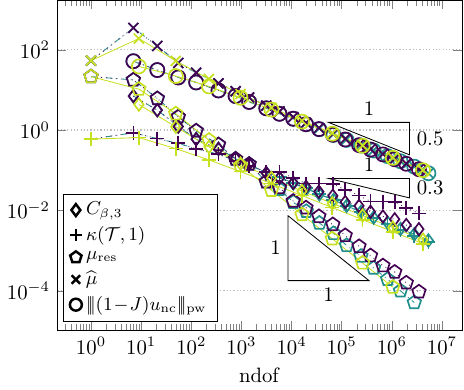}
		\includegraphics{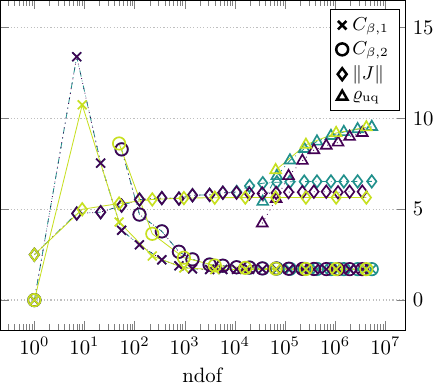}\\
		\caption{Convergence history plot of several quantities of \cref{tab:sol_verification}
			for $\lambda=100$ in
			\cref{sec:Academic_benchmark_infsup}.
		 }
		\label{fig:Square_Morley_InfSup_Const}
	\end{figure}
	
	\Cref{fig:Square_Morley_InfSup_Const} displays the convergence history of several terms from
	\cref{tab:sol_verification} computed from the Morley
	solution $v_\nc=u_\nc$ and its smoothed version $\ut=Ju_\nc$ as described in
	\cref{sec:The_continuous_inf--sup_constant,sec:Algebraic_realisation_and_application}.
	The residual estimator $\mu_{\rm res}=\mu_{\rm res}(\mathcal{T})$ 
	from \cref{lem:explicit_constants_NVS}.a 
	converges with higher order, as expected by the mesh size factors.
	This, the convergence $L_G\to0$ (undisplayed), and $\operatorname{Res}_h=0$ from an exact solve 
	(up to machine precision) justify the observed supercloseness of $\trbpw{(1-J)u_\nc}$ and
	$\mut$ from~\eqref{eqn:mut_N2_def}.
	The insufficient reduction of $h_{\max}$ from the standard AFEM leads to a reduced 
	experimental convergence rate~$0.3$
	of the approximation error control $\kappa(\T, 1)$ from~\eqref{eqn:kappa_Ts}.
	Uniform refinement or the modified adaptive strategy with a guaranteed reduction of the mesh size results in the
	optimal rate~$0.5$.
	The convergence of $\Constb{3}$, 
	displayed in \cref{fig:Square_Morley_InfSup_Const} with rate $0.5$ enables
	the observed convergence of the lower bounds $\btT$ and $\beta_0$
	to the continuous inf--sup constant $\beta$.
	The reference value for $\beta$ is obtained by Aitken extrapolation of $\bht$ for a
	sequence of uniformly refined meshes.
	It is displayed in \cref{tab:NVS_Square_infsup} together with the guaranteed lower
	bounds computed, and the existence and uniqueness radii
	on fine meshes with more than $4\times10^6$ ndof.

	\begin{figure}[!h]
		\centering
		\includegraphics{./Figures/NVS/Legend_uniform_adaptive_color_out.pdf}\\
		\includegraphics{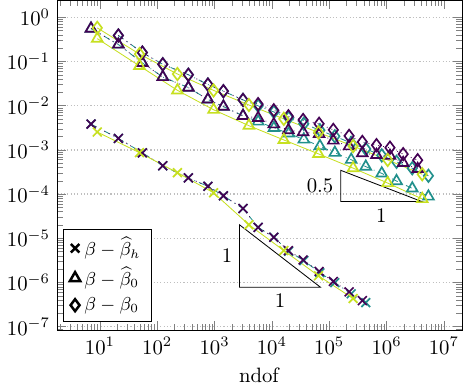}
		\includegraphics{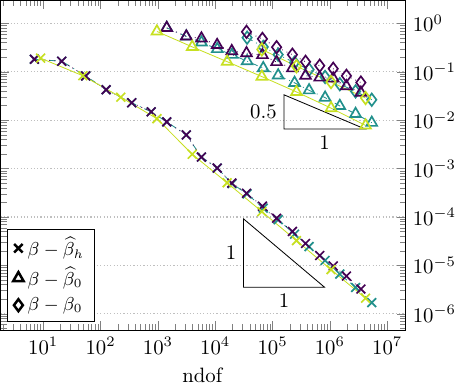}
		\caption{Convergence history of the (lower bounds on the) inf--sup constants 
			under the three refinement strategies for $\lambda=1$ (left)
			and $\lambda=100$ (right) in \cref{sec:Academic_benchmark_infsup}.
		 }
		\label{fig:Square_Morley_infsup_Convergence}
	\end{figure}

	\begin{table}[ht]
		\centering
		\begin{tabular}{l|lll}
			$\lambda$ &	$1$			&	$10$			&	$100$			\\\hline
			$\beta$&	$0.9999209$	&	$0.9992118$		&	$0.9924483$		\\
			$\btT$&	$0.9998327$		&	$0.9983298$		&	$0.9835772$		\\
			$\beta_0$&	$0.9996614$	&	$0.9965771$		&	$0.9659684$		\\\hline
			$\vrpT$&$9.866264\phantom{\; (\infty)}$&$9.835858\phantom{ \; (\infty)}$&
										$9.659684 \phantom{\; (\infty)}$		\\
			$\vrmT$&	$0.0016892$	&	$0.0172620$		&	$0.1712798$
		\end{tabular}
		\caption{Lower bound $\beta_0$ on the inf--sup constant $\beta$ of $DN(u)$,
		lower bound $\btT$ on the inf--sup constant of $DN(Ju_\nc)$, radius of uniqueness 
		$\vrpT$, and error bound
		$\vrmT$ computed on meshes with more than $4\times10^6$ ndof in 
		\cref{sec:Academic_benchmark_infsup}.}
		\label{tab:NVS_Square_infsup}
	\end{table}

	\subsection{Singular solution on the L-shaped domain}%
	\label{sec:Singular_solution_to_the_Navier_Stokes}
	This benchmark considers the L-shaped domain $\Omega=(-1,1)^2\setminus[0,1)^2$ 
	from \cref{fig:initial_mesh} with
	$\sigma_{\rm reg}\equiv\mu=0.54448$.
	The source $F$ matches the singular function 
	from~\cite{Gri:SingularitiesBoundaryValue1992},
	given by
	\begin{align*}
		u(r,
		\varphi)=\big(r^2\sin(\varphi)^2-1\big)^2\big(r^2\cos(\varphi)^2-1\big)^2r^{1+\mu}\;\xi\big(\varphi-\pi/2\big),
	\end{align*}
	in polar coordinates, with smooth
	$\xi$ from~\cite[eqn.~3.2.9]{Gri:SingularitiesBoundaryValue1992} 
	(therein $\xi(\varphi) = u(\mu,\varphi)$).

	\Cref{fig:Lshape_Morley_infsup} displays the optimal convergence rate $0.5$ for
	the exact error $\trbpw{u-u_\nc}$ and the error estimator $\eta$ under both
	adaptive strategies.
	The suboptimal empirical rate $\sigma_{\mathrm{reg}}/2=0.27$ under uniform refinement is 
	expected by the singularity of $u$ at the origin.
	The singularity also explains the same experimental convergence rate of 
	$\trbpw{(1-J)u_\nc}$ and, as predicted by \cref{thm:discretisation_error},
	$\kappa(\T, 1)\lesssim h_{\max}^{\sigma_{\rm reg}}$ on uniformly refined meshes.
	An insufficient reduction of the maximal mesh size $h_{\rm max}\lesssim \kappa(\T,1)$
	leads to a reduced convergence rate $0.3$ for $\kappa(\T,1)$ in
	the standard AFEM loop,
	while the modified adaptive refinement strategy recovers the optimal rate $0.5$.
	The influence on the lower bounds $\btT$ from \cref{thm:continuous_inf_sup} is significant: 
	$\btT$ converges significantly faster for the latter strategy and
	enables a better guaranteed lower bound $\beta_0$ on the inf--sup constant of $DN(u)$ 
	on coarser meshes.
	Computations on meshes with more than $9\times10^6$ ndof reveal
	\begin{align*}
		\beta=0.96463412,\;\btT=0.94340634,\;\beta_0=0.87147780,\;\vrmT=0.38937903,\;\vrpT=4.98635670
	\end{align*}
	(with Aitken extrapolation $\beta$ of $\bht$ from a sequence of uniformly refined triangulations)
	and guarantee local uniqueness of the exact
	solution $u$ in a ball of radius $\vrpT$ around $u_\nc$.

	\begin{figure}[!h]
		\centering
		\includegraphics{./Figures/NVS/Legend_uniform_adaptive_color_out.pdf}\\
		\includegraphics{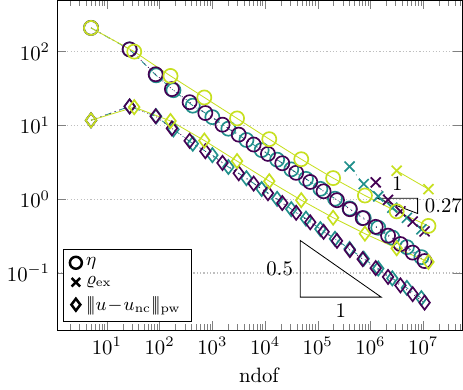}
		\includegraphics{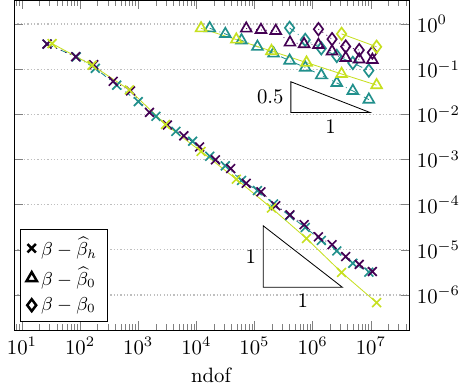}
		\includegraphics{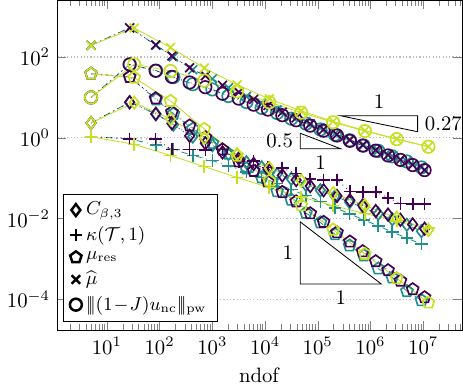}
		\includegraphics{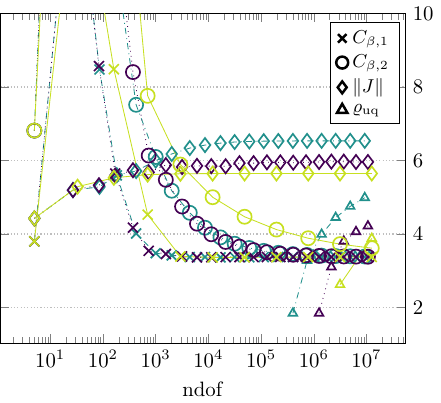}
		\caption{Convergence history plot of different error measures (top left), the (lower bounds of the) inf--sup constants
			(top right), and several quantities from \cref{tab:sol_verification} (bottom) in
			\cref{sec:Singular_solution_to_the_Navier_Stokes}.
		 }
		\label{fig:Lshape_Morley_infsup}
	\end{figure}

	\appendix
	\section[Proof of theorem 2.5]{Proof of \cref{thm:abstract_apriori}}%
	\label[appendix]{apx:Proof of a priori}
	This appendix presents a self-contained and simplified proof of \cref{thm:abstract_apriori} extending the arguments in \cite{CNRS:UnifiedPrioriAnalysis2023} beyond quadratic semilinearities in three steps.

	\medskip\noindent\emph{Step 1} prepares the proof.
	By assumption, $DG_\pw: V_\pw\to L(V_\pw;V^*)$ is locally Lipschitz at $u\in V$: 
	There exists $\delta_u, L_u>0$ such that 
	\begin{align*}
		\|DG_\pw(v_\pw)-DG_\pw(w_\pw)\|_{L(V_\nc;V^*)}
		&\leq  L_u \trbpw{v_\pw-w_\pw}
	\end{align*}
	for all $v_\pw,w_\pw\in B_\pw(u,\delta_u)\coloneqq\{v_\pw\in V_\pw\ :\ \trbpw{u-v_\pw}\leq \delta_u\}$. 
	Recall $\trbpw{(1-I)u}\leq \delta_0$ from~\eqref{eqn:H2_H4} and require $\delta_0\leq \delta_u/2$ from hereon such that $B_\pw(Iu,\delta_u/2)\subset B_\pw(u,\delta_u)$.
	This and~\eqref{eqn:DSWP} imply the local Lipschitz continuity of $DN_h$ with constant $L_u\|J\|$ in $B_\pw(Iu,\delta_u/2)\cap V_\nc$.
	Moreover, $Iu\in B_\pw(u,\delta_u)$ implies
	\begin{align}
		\|DG_\pw(u)-DG_\pw(Iu)\|_{L(V_\nc;V^*)}
		&\leq L_u\delta_0.\label{eqn:loc_Lip_DG}
	\end{align}
	The analogue estimate for the locally Lipschitz $G_\pw\in C^1(V_\pw;V^*)$ with $\widetilde L_u>0$ reads
	\begin{align}\label{eqn:loc_Lip_G}
		\trb{G_\pw(u)-G_\pw(Iu)}_*
		\leq \widetilde L_u\delta_0.
	\end{align}

	\medskip\noindent\emph{Step 2} controls the discrete inf--sup constant at $Iu$. 
	Consider any $v_\nc\in V_\nc$ with $\trbpw{v_\nc}=1$ and define $\xi=A^{-1}DG_\pw(u;v_\nc)\in V$ 
	as the unique solution to $a(\xi,\bullet )= DG_\pw(u;v_\nc)\in V^*$. 
	The regular root $u\in V$ has a positive inf--sup constant $\beta>0$ and%
	~\eqref{eqn:cont_inf_sup} provides
	\begin{align*}
		\beta\trb{Jv_\nc}\leq \trb{DN(u;Jv_\nc)}_*\leq \trb{Jv_\nc + \xi}+\|DG_\pw(u)\|_{L(V_\pw;V^*)}\trb{(1-J)v_\nc}_\pw
	\end{align*}
	with the boundedness of $DG_\pw(u)\in L(V_\pw;V^*)$ in the last step.
	This, triangle inequalities, and the quasi-optimality~\eqref{eqn:J_qo} of $J$ reveal
	\begin{align*}
		1=\trbpw{v_\nc}\leq\trbpw{(1-J)v_\nc}+\trbpw{Jv_\nc}
		\leq c_0^{-1}/2\trbpw{v_\nc+\xi}
	\end{align*}
	with $c_0\coloneqq(2\LJ+2\beta^{-1}(1+(1+\|DG_\pw(u)\|_{L(V_\pw;V^*)})\LJ))^{-1}$.
	The orthogonality~\eqref{eqn:I1} and $IJ\varphi_\nc=\varphi_\nc$ from~\eqref{eqn:JI} for $\varphi_\nc\coloneqq (v_\nc+I\xi)/\trbpw{v_\nc+I\xi}$ 
	verify with $\trbpw{\varphi_\nc}=1$ that
	\begin{align*}
		\trbpw{v_\nc+I\xi}
		&=a_\pw(v_\nc+\xi,J\varphi_\nc) - a_\pw((1-I)\xi,J\varphi_\nc)\\
		&\leq a_\pw(v_\nc+\xi,J\varphi_\nc) + \|J\|\trbpw{(1-I)\xi}.
	\end{align*}
	Algebraic manipulations with the definition of $\xi\in V$ and~\eqref{eqn:DSWP} reveal 
	\begin{align*}
		a_\pw(v_\nc+\xi,J\varphi_\nc)
		&=DN_h(Iu; v_\nc,\varphi_\nc) + DG_\pw(u;v_\nc,J\varphi_\nc) - DG_\pw(Iu; v_\nc, J\varphi_\nc)\\
		&\leq DN_h(Iu; v_\nc,\varphi_\nc) + L_u\|J\|\delta_0
	\end{align*}
	with~\eqref{eqn:loc_Lip_DG} and $\trbpw{v_\nc}=\trbpw{\varphi_\nc}=1$ in the last step.
	The previous displayed estimates combine with a triangle inequality and
	$\trbpw{(1-I)\xi}\leq \delta_0$ from~\eqref{eqn:H2_H4} to
	\begin{align*}
		2c_0 \leq {DN_h(Iu;v_\nc,\varphi_\nc)}+ (1+(1+L_u)\|J\|)\delta_0.
	\end{align*}
	For $\delta_0\leq c_0/(1+(1+L_u)\|J\|)$ sufficiently small, this shows
	\begin{align*}
		c_0
		\leq \inf_{0\ne v_\nc\in
		\Vh}\sup_{0\ne \varphi_\nc\in\Vh}\frac{DN_h(Iu;v_\nc,\varphi_\nc)}{\trbpw{v_\nc}\trbpw{\varphi_\nc}}.
	\end{align*}

	\noindent\emph{Step 3} concludes the proof. 
	Similar to \cref{lem:cont_inf_sup}, the discrete inf--sup constant from step~2 provides the invertibility of $DN_h(Iu)$ in the finite-dimensional space $V_\nc$ with 
	\begin{align*}
		\|DN_h(Iu)^{-1}\|_{L(V_\nc^*;V_\nc)}\leq m\coloneqq c_0^{-1}.
	\end{align*}
	Moreover,~\eqref{eqn:DSWP} and $N(u) = 0$ in~\eqref{eqn:ASWP} imply for any $\varphi_\nc\in V_\nc$ with $\trbpw{\varphi_\nc}=1$ that
	\begin{align*}
		N_h(Iu,\varphi_\nc) &= a_\pw(Iu-u,J\varphi_\nc) + G_\pw(Iu; J\varphi_\nc) - G(u; J\varphi_\nc),\\
		&\leq(1+ \widetilde L_u)\|J\|\trbpw{(1-I)u}
	\end{align*}
	with~\eqref{eqn:loc_Lip_G} in the last step.
	This verifies
	\begin{align*}
		\trbpw{DN_h(Iu)^{-1}N_h(Iu)}\leq \varepsilon\coloneqq c_0^{-1}(1+ \widetilde L_u)\|J\|\trbpw{(1-I)u}.
	\end{align*}
	Set $L\coloneqq\max\{L_u\|J\|, 2/(m\delta_u)\}$ and $r\coloneqq 1/(Lm)\leq \delta_u/2$, and recall from step~1 that $DN_h$ is Lipschitz continuous with Lipschitz constant $L$ in $B_\pw(Iu,r)$.
	Since $\varepsilon\lesssim \trbpw{(1-I)u}\leq\delta_0$, $L m \varepsilon\leq 1/2$ for sufficiently small $\delta_0$ and the Newton--Kantorovich theorem~\cite[thm.~3]{CM:NewtonKantorovichTheorem2012} provides a discrete root $u_\nc\in V_\nc$ of $N_h(u_\nc)=0$ that is unique in $B_\pw(Iu,r)\cap V_\nc$.
	Using $1-\sqrt{1-t}\leq t$ for all $t\in[0,1]$, the Newton--Kantorovich a priori bound reveals
	\begin{align*}
		\trbpw{Iu-u_\nc}\leq \frac{1-\sqrt{1-2Lm\varepsilon}}{Lm}\leq2\varepsilon\lesssim\trbpw{(1-I)u}.
	\end{align*}
	Hence a triangle inequality and the best-approximation property of $Iu$ by~\eqref{eqn:I1} provide 
	\begin{align*}
		\trbpw{u-u_\nc}\lesssim \trbpw{(1-I)u}=\min_{v_\nc\in V_\nc}\trbpw{u-v_\nc}.
	\end{align*}
	This is the quasi-best approximation property and choosing $\delta_0\leq r/2$ small enough, one can guarantee that $u_\nc\in V_\nc$ is the only discrete root of $N_h(u_\nc)=0$ with $\trbpw{u-u_\nc}\leq \varepsilon_0\coloneqq r/2$. 
	The uniform bound on the inf--sup constant follows as in step 2 in the proof of \cref{thm:newton_kantorovich} and concludes this proof.\qed

	\bibliographystyle{alphaabbr}
	\bibliography{Bibliography}

\end{document}